\documentclass[11pt,reqno]{amsart}

\usepackage{amssymb,amsthm,amsmath,mathtools}
\usepackage{mathdots,bm}
\usepackage[left=2.5cm, right=2.5cm, top=3cm, bottom=3cm]{geometry}
\usepackage{titlesec}
\titleformat{\subsection}[hang]{\normalfont\bfseries}{\thesubsection}{1em}{}
\usepackage{secdot}
\usepackage{ifthen}
\usepackage{mathabx}
\usepackage{hyperref}
\usepackage{float}
\usepackage{import}
\usepackage{graphicx}   
\usepackage{listings}
\usepackage{color} 
\usepackage{xcolor}
\definecolor{codegreen}{rgb}{0,0.6,0}
\definecolor{codegray}{rgb}{0.5,0.5,0.5}
\definecolor{codepurple}{rgb}{0.58,0,0.82}
\definecolor{backcolour}{rgb}{0.95,0.95,0.92}
\lstdefinestyle{mystyle}{
    backgroundcolor=\color{backcolour},   
    commentstyle=\color{codegreen},
    keywordstyle=\color{magenta},
    numberstyle=\tiny\color{codegray},
    stringstyle=\color{codepurple},
    basicstyle=\ttfamily\footnotesize,
    breakatwhitespace=false,         
    breaklines=true,                 
    captionpos=b,                    
    keepspaces=true,                 
    numbers=left,                    
    numbersep=6pt,                  
    showspaces=false,                
    showstringspaces=false,
    showtabs=false,                  
    tabsize=4
}

\usepackage{multirow}  
\usepackage{float}
\usepackage{tcolorbox}
\usepackage[numbered]{matlab-prettifier}
\usepackage{diagbox}
\usepackage{cleveref}
\usepackage{diagbox}
\usepackage{algorithm}
\usepackage{algpseudocode}
\usepackage{caption}
\usepackage{subcaption}
\usepackage{textcomp,multirow}
\theoremstyle{plain}
\newtheorem{thm}{Theorem}[section]
\newtheorem*{thm*}{thm}

\numberwithin{equation}{section}

\newtheorem{lem}{Lemma}[section]
\newtheorem{prop}{Proposition}[section]

\newtheorem{rem}{Remark}[section]

\theoremstyle{definition}
\newcounter {own}
\def\theown {\thesection  .\arabic{own}}

\newenvironment{pf}[1][]{%
 \vskip 2mm
 \noindent
 \ifthenelse{\equal{#1}{}}%
  {{\slshape Proof. }}%
  {{\slshape #1.} }%
 }%
{\qed\bigskip}
\graphicspath{ {./images/} }

\DeclareMathOperator*{\supp}{\supp}
\DeclareMathOperator*{\rank}{rank}

\newcounter{alphabet}

\newcommand{\R}{\mathbb{R}}

\newcommand{\N}{\mathbb{N}}
\newcommand{\Z}{\mathbb{Z}}

\newcommand{\qtq}[1]{\quad\text{#1}\quad}

\newtheorem{theorem}{Theorem}[section]

\newtheorem{conjecture}[theorem]{Conjecture}
\theoremstyle{definition}

\theoremstyle{remark}

\newcommand{\ds}{\displaystyle}
\begin{document}
%\bibliographystyle{amsplain}
%bibliographystyle{abbrvnat}
%\bibliography{res_poly_curve}
\title[Gabor Frame Regions and Non-Frame Obstructions for an even Rational Window]{Gabor Frame Regions and Non-Frame Obstructions for an even Rational Window}
\thanks{Indian Institute of Technology (IIT) Bombay, Mumbai, India, 400076\\
Email: riya74012@gmail.com}
\author{Riya Ghosh}


\begin{abstract}
We study the Gabor frame properties of the rational window
$$ g(x)=\frac{x^2-1}{(x^2+1)(x^2+4)(x^2+9)}, $$
whose poles occur in symmetric pairs. We prove that every lattice satisfying $0<\alpha\beta<1/3$ generates a frame, and we establish an additional frame region for $\beta\ge1$ and $1/3\le \alpha\beta<0.47373$. Furthermore, we show that the rational hyperbolas $\alpha\beta=\frac{p}{3p-1}$, for $p\ge2,$ are entirely contained within the frame set. In contrast, we construct explicit non-frame lattice points on the hyperbolas
$$\alpha\beta\in\left\{\frac{1}{3},\frac{1}{2},\frac{2}{3},\frac{3}{4}\right\}.$$ 
Finally, for the density family $\alpha\beta=p/(p+1)$, we derive a symmetry reduction of the associated Zibulski–Zeevi matrix, providing numerical evidence for a richer structure of non-frame obstructions.
\end{abstract}
\date{\today}
\maketitle
\pagestyle{myheadings}
\markboth{Riya Ghosh}{Gabor Frames for an even Rational Window with Symmetric Poles}
\section{introduction}
\label{introduction}
Let $g\in L^2(\mathbb{R})$ be a non-zero window function with lattice parameters $\alpha, \beta > 0$. The corresponding Gabor system  $\mathcal{G}(g,\alpha, \beta)=\{e^{2\pi i\beta m\cdot}g(\cdot-\alpha k): k, m\in\mathbb{Z}\}$ is called a Gabor frame for $L^2(\mathbb{R})$ if there exist two positive constants $A$, $B$ such that 
\begin{equation}\label{Gaborframe}
A\|f\|^2\leq\displaystyle\sum_{m,k\in\mathbb{Z}}|\langle f,e^{2\pi i\beta m\cdot} g(\cdot-\alpha k)\rangle|^2\leq B\|f\|^2,
\end{equation}
for every $f\in L^2(\mathbb{R})$. The constants $A$ and $B$ are called frame bounds. The set of all such lattice parameters is referred to as the frame set of $g$ and is given by
$$\mathcal{F}(g)= \left\{(\alpha, \beta) \in\mathbb{R}^2_+ : \mathcal{G}(g, \alpha,\beta)~ \text{is a frame}\right\} .$$
In this paper, we use the following version of the Fourier transform: 
$$\widehat f(w)= \int_{-\infty}^{\infty} f(x) e^{-2\pi \mathrm{i}wx}dx,~ w\in\mathbb{R}.$$
Feichtinger and Kaiblinger \cite{HGF1} proved that $\mathcal{F}(g)$ is an open subset of $\mathbb{R}^2_+$ for a window $g$ in the Feichtinger algebra. The fundamental density theorem asserts that 
$$\mathcal{F}(g)\subseteq \{(\alpha, \beta) \in\mathbb{R}^2_+: \alpha\beta\le 1\}$$ 
(see \cite{time, hedtg}). In addition, if $g$ is in the Feichtinger algebra, the Balian-Low \cite{balian, rdbalian} theorem states that 
$$\mathcal{F}(g)\subseteq \{(\alpha, \beta) \in\mathbb{R}^2_+: \alpha\beta< 1\}.$$
For a more comprehensive discussion on Gabor analysis, we refer to \cite{ole, time}. 

The determination of the frame set of a given window is a difficult problem, and complete descriptions are known only for a limited number of generators. Important examples include the Gaussian $e^{-\pi x^2}$ \cite{fblyu,dtsibf2}, the hyperbolic secant \cite{hsyg}, the one-sided exponential $e^{-x}\chi_{[0,\infty)}(x)$ \cite{whfb}, the two-sided exponential $e^{-|x|}$ \cite{tgfcd}, characteristic functions $\chi_{[0,c)}$, $c>0$, \cite{abc, when}, totally positive functions \cite{stsis, duke,pont2026gabor}, the Haar function \cite{Haar}, the first Hermite function \cite{faulhuber2026frame}, and rational Herglotz functions whose poles lie entirely in one half-plane \cite{Yurii:rational:2023}. Beyond these complete descriptions, several works have obtained substantial partial frame regions for broader classes of generators. General background on Gabor frames and frame-set problems can be found in \cite{ole, feichtinger2012advances, feichtinger2012gabor,grochenig2014mystery}.

The present work is motivated by the fact that the Gabor frame behavior changes substantially when one passes from an odd to an even rational window with symmetrically placed poles. In our previous study \cite{ghosh2026gabor}, the generator was odd, and the oddness itself imposed a strong obstruction by the Lyubarskii--Nes \cite{lyubarskii2013gabor}, the Gabor system fails to be a frame whenever
$$ \alpha\beta=\frac{p}{p+1},\quad p\in\mathbb N. $$
Thus, for the odd window, a substantial part of the non-frame structure is already forced by symmetry.

In the present paper, we turn to the even rational window
\begin{align}\label{window}
g(x)=\frac{x^2-1}{(x^2+1)(x^2+4)(x^2+9)}.
\end{align}
Its poles again occur in symmetric pairs, namely $\pm ik$, $k=1,2,3$, and $g\in \mathcal S_0(\mathbb R)\cup W(L^\infty,\ell^1)$. Despite this similar pole configuration, the frame-set problem changes substantially. The obstruction available in the odd case disappears, so neither frame nor non-frame behavior is dictated by parity alone. This raises two natural questions: whether the symmetric pole structure still gives rise to nontrivial frame regions, and whether entire rational hyperbolas can be contained in the frame set. The results of this paper show that all of these phenomena indeed occur.

The frame-set problem for the window \eqref{window} is not covered by
several existing complete characterizations. Belov \textit{et al.}
\cite{Yurii:rational:2023} treated rational Herglotz functions whose
poles lie in a common half-plane, while Ulanovskii and Zlotnikov
\cite{ulanovskii2025sampling} considered generators given by ratios of
exponential polynomials. In contrast, \eqref{window} is an algebraic
rational function with symmetric poles at $\pm ik$, $k=1,2,3$. Semenov
\cite{semenov2025} recently established a frame result for a broad
class of rational functions, including the present window, for suitably
constructed nonuniform frequency sets. This does not, however,
determine the frame set for the rectangular lattices considered here.

A central difficulty is that the frame behavior of the even window is not determined by the rational density $\alpha\beta$ alone. Although Kulikov \cite{kulikov2025} showed that the frame set of a rational window is relatively open along every fixed-density hyperbola, such a hyperbola need not be entirely contained in the frame set. For the present window, the hyperbola $\alpha\beta=\frac12 $
contains both frame and non-frame lattice points. This makes it necessary to study the actual two-parameter geometry of the frame set rather than only the density parameter.

The main contributions of the paper reflect this distinction. We first establish new explicit regions contained in the frame set and prove that the hyperbolas
$\alpha\beta=\frac{p}{3p-1}$, for $p\ge2$
form an infinite family of frame hyperbolas. In contrast, we construct explicit non-frame points on the hyperbolas
$\alpha\beta\in \left\{ \frac13,\frac12,\frac23,\frac34 \right\},$
showing that these densities do not give universal frame curves.

For the family
$$ \alpha\beta=\frac{p}{p+1}$$ for $p\in \N$
where the odd-window obstruction would have been automatic, we instead derive a new symmetry reduction of the Zibulski--Zeevi matrix. This produces finite square blocks whose determinant vanishing forces rank loss. Numerical computations based on this reduction indicate obstruction points for $2\le p\le200$, and reveal multiple obstruction branches; notably, one branch reaches the region $\beta>1$.

These results show that passing from an odd to an even rational generator is not a minor variation of the same frame-set problem. The parity change removes the known universal obstruction, modifies the symmetry of the associated polynomial structure, and leads to a more delicate interaction between frame and non-frame behavior. This provides the main reason for studying the even window independently. We summarize the principal results below. Figure \ref{fig:Figure1} illustrates all results.

\begin{thm}\label{frame_region0}
The Gabor system $\mathcal G(g,\alpha,\beta)$ forms a frame for $L^2(\R)$ whenever $0<\alpha\beta<\frac{1}{3}$. 
\end{thm}

\begin{thm}\label{frame_region1}
Let $a^*\in \left(\frac{1}{3},\frac{1}{2}\right)$ be the unique solution of 
$$\frac{ 2e^{-2\pi(1/x-2)} }{ (1-e^{-2\pi/x})(1-e^{-2\pi})^2 }=1.$$
Then $\mathcal G(g,\alpha,\beta)$ forms a frame for $\beta\ge1$ and $\frac{1}{3}\le \alpha\beta<a_*$. 
\end{thm}

\begin{thm}\label{frame_region2}
$\mathcal G(g,\alpha,\beta)$ forms a frame along the hyperbolas 
$\alpha\beta=\frac{p}{3p-1},$
for every integer $p\ge2$.
\end{thm}

\begin{thm}\label{non_frame_1}
For each 
$$a\in \left\{\frac{1}{3},\frac{1}{2},\frac{2}{3},\frac{3}{4}\right\},$$
there exists a lattice point $(\alpha,\beta)\in \mathbb{R}_+^2$ satisfying $\alpha\beta=a$
such that $\mathcal G(g,\alpha,\beta)$ does not form a frame for $L^2(\mathbb R)$.
\end{thm}

The remainder of the paper is organized as follows. Section \ref{section_2} develops the algebraic and structural tools used throughout the paper. We derive the reciprocal symmetry of the generating polynomial, establish sign and zero properties of the coefficient functions through Wronskian estimates and extended complete Chebyshev systems, and recall the Zak transform and Zibulski--Zeevi formulations of the frame condition. Section \ref{Section_3} is devoted to the positive frame results; we prove the new frame regions in Theorems \ref{frame_region0} and \ref{frame_region1}, the universal frame hyperbolas in Theorem \ref{frame_region2}, and show in addition that $(1/2,1)$ is a frame point. Section \ref{Section_4} treats the complementary non-frame phenomena. We first establish the critical-density obstruction and the explicit obstruction points of Theorem \ref{non_frame_1}. We then study the family $\alpha\beta=p/(p+1)$ through a symmetry reduction of the Zibulski--Zeevi matrix and conclude with numerical evidence for the resulting determinant condition, including the appearance of a second obstruction branch for large $p$.

\begin{figure}[H]
    \centering
    \includegraphics[width=6 in]{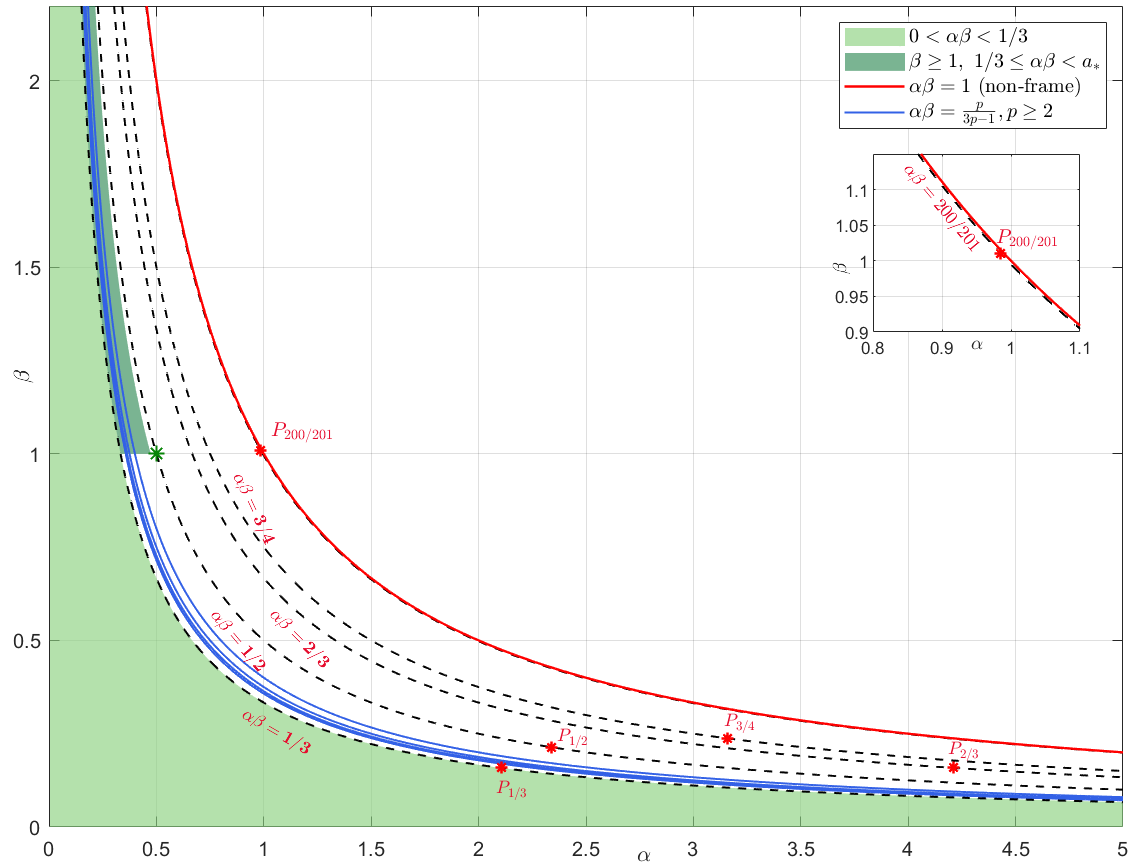}
    \caption{The light-green region corresponds to $0<\alpha\beta<1/3$, while the darker region indicates the additional frame region $\beta\ge1$ and $1/3\le\alpha\beta<a^\ast$, where $a^\ast\approx0.47373$ together with irrational densities $0<\alpha\beta<1$ in \cite{Yurii:rational:2023}. The blue hyperbolas $\alpha\beta=\frac{p}{3p-1}$, for $p\ge2$
are universal frame curves. The red marked points $P_{1/3},P_{1/2},P_{2/3},P_{3/4}$ are explicit non-frame obstruction points. The point \(P_{200/201}\) is a numerically detected obstruction point on the hyperbola $\alpha\beta=\frac{200}{201}$. It shows that the corresponding hyperbolas are not universal frame curves. Further, the star point $(1/2,1)$ forms a frame.}
    \label{fig:Figure1}
\end{figure}

\section{Algebraic and Structural Preliminaries}\label{section_2}
The partial-fraction decomposition of $g$ plays a central role in our analysis.
Let
\begin{align}\label{defn w_k}
    \omega_k=k,\qtq{and} \omega_{3+k}=-\omega_k,\qtq{for} 1\le k\le 3.
\end{align}
We can rewrite \eqref{window} as
\begin{align}\label{rational}
g(x)=\sum_{k=1}^{6}\frac{b_k}{x-i\omega_k},
\end{align}
where,
\begin{align}\label{defn a_k}
b_1=b_3=\frac{i}{24},~ b_2=-\frac{i}{12},~b_4=b_6=-b_1\text{ and } b_5=-b_2.
\end{align}
Here, 
$\sum_{k=1}^{6} b_k=0$ with $\sum_{k=1}^{6}\omega_k=0$.
Let $\tau=2\pi\beta$ throughout this paper.
Define 
\begin{align}\label{Defn Aks}
  \mathrm A_{k, s}=
    (- 1)^s \ds\sum\limits_{\substack{1\le j_1 < \cdots < j_s\le 6\\ j_l \ne k}} e^{\tau (\omega_{j_1} + \cdots + \omega_{j_s})}.
\end{align}
For $0\le s\le 5$, we consider the function
\begin{align}\label{E : def m_S}
m_{s,\beta}(x) = \sum_{k = 1}^{6} \widetilde b_{k}\, \mathrm A_{k, s} \, e^{\tau x\omega_k }, \quad x \in \R,
\end{align}
where $\widetilde{b}_k=-12 i b_k$ and $b_k$s are defined in \eqref{defn a_k} for $1\le k\le 6$. Using \eqref{defn a_k}, we have
$$ m_{0,\beta}( x) = \sinh(\tau  x)-2\sinh(2\tau  x)+\sinh(3\tau  x)= 4\sinh(2\tau  x) \sinh^2\!\left(\frac{\tau  x}{2}\right)>0, \quad  x>0.$$
Since $m_{0,\beta}(x)\ne 0$ for $x>0$ and $\omega_k\ne \omega_l$, then the Gabor system $\mathcal{G}(g, \alpha,\beta)$ is a frame for $L^2(\R)$ whenever $\alpha\beta<1$ is irrational (see \cite{Yurii:rational:2023}). Furthermore, it is established in \cite[Theorem 1.9]{Yurii:rational:2023} that $\left\{(\alpha,\beta):~\alpha\beta\le \frac{1}{6}\right\}\subset \mathcal{F}(g)$. In this paper, therefore, we focus primarily on rational lattices for $0<\alpha\beta<1$.

Let $\widetilde{\mathrm A}_{s}=\sum\limits_{1\le j_1<\cdots<j_s\le 6} e^{\tau (\omega_{j_1}+\cdots+\omega_{j_s})}$ with $\widetilde{\mathrm A}_{0}=1$. It is easy to check that
\begin{align}\label{Aks and tilde As}
 \mathrm A_{ks} = e^{\tau \omega_k}\, \mathrm A_{k,s-1} + (-1)^s \widetilde{\mathrm A}_{s}.
\end{align}
Using \eqref{Aks and tilde As}, we obtain
\begin{align}
    m_{1,\beta}(x)&=m_{0,\beta}(x+1)-\widetilde{\mathrm{A}}_1 m_{0,\beta}(x),\nonumber\\ 
    m_{2,\beta}(x)&= m_{0,\beta}(x+2)-\widetilde{\mathrm{A}}_1 m_{0,\beta}(x+1)+\widetilde{\mathrm{A}}_2m_{0,\beta}(x),\nonumber
\end{align}
where
\begin{align}\label{A1_eqn}
\widetilde {\mathrm{A}}_1 = 2(\cosh\tau+\cosh2\tau+\cosh3\tau),
\end{align}
and
\begin{align}\label{A2_eqn}
\widetilde {\mathrm{A}}_2= 3+4\cosh\tau +2\cosh2\tau+2\cosh3\tau +2\cosh4\tau+2\cosh5\tau.
 \end{align}
The polynomial
\begin{align}\label{P_Ndefinition}
\mathcal{P}_\beta(z; x)=\sum_{l=0}^{5}m_{l,\beta}( x)z^l=\mathcal R_\beta(z)\, \mathcal Q_\beta(z\,; x),
\end{align}
where
$$\mathcal R_\beta(z)=\prod_{k=1}^{6} (1 - z e^{\tau\omega_k}),\quad\mathcal Q_\beta(z\,; x)=\sum_{k=1}^{6} \frac{\widetilde b_{k}\,e^{\tau x \omega_k}}{1 - z e^{\tau \omega_k}},$$
and $\omega_k$ are defined in \eqref{defn w_k}.
Setting $\gamma_k=e^{\tau k}$ yields $\gamma_{3+k}=\gamma_k^{-1}$ allowing
\begin{align}\label{R_beta}
    \mathcal R_\beta(z)=\prod_{k=1}^{3}\left(1-z\gamma_k\right)\left(1-z\gamma_k^{-1}\right),
\end{align}
    which explicitly satisfies the self-reciprocal identity
    \begin{align}\label{reciprocal_R}
        z^{6} \mathcal R_\beta(1/z)=\prod_{k=1}^{3}\left(z-\gamma_k\right)\left(z-\gamma_k^{-1}\right)=\mathcal R_\beta(z).
    \end{align}

To analyze the rank of the Zibulski-Zeevi matrix, we need to understand the structure of its generating polynomial, $\mathcal{P}_\beta(z; x)$. The following lemma establishes a fundamental reciprocal symmetry connecting the polynomial evaluated at $x$ with its evaluation at $1-x$.

\begin{lem} \label{poly_reciprocity}
For $\beta>0$, $ x \in [0,1]$ and $z\ne 0$,
\begin{equation}\label{eq:poly_reciprocity}
z^{5}\mathcal{P}_\beta\left(\frac{1}{z}; x\right) = \mathcal{P}_\beta(z;1- x).
\end{equation}
\end{lem}

\begin{pf}
Let $\gamma_{k}=e^{\tau\omega_k}$, $\tau=2\pi\beta$. Now
\begin{align}
\mathcal Q_\beta\left(\frac{1}{z}; x\right) &= \sum_{k=1}^{6} \frac{\widetilde b_k e^{\tau x\omega_k}}{1 - \frac{1}{z}\gamma_k} = \sum_{k=1}^{6} \frac{z \widetilde b_k e^{\tau x\omega_k}}{z - \gamma_k}\nonumber\\
&= \sum_{k=1}^{3} \frac{z \widetilde b_k e^{\tau x\omega_k}}{z - \gamma_k} - \sum_{k=1}^{3} \frac{z \widetilde b_k e^{-\tau x \omega_k}}{z - \gamma_k^{-1}}\nonumber\\
&= \sum_{k=1}^{3} \frac{z \widetilde b_k \gamma_k^{-1} e^{\tau x\omega_k}}{z\gamma_k^{-1} - 1} - \sum_{k=1}^{3} \frac{z \widetilde b_k \gamma_k e^{-\tau x \omega_k}}{z\gamma_k - 1}\nonumber\\
&= -\sum_{k=1}^{3} \frac{z \widetilde b_k e^{-(1- x)\tau\omega_k}}{1 - z\gamma_k^{-1}} + \sum_{k=1}^{3} \frac{z \widetilde b_k e^{(1- x)\tau\omega_k}}{1 - z\gamma_k}.\nonumber
\end{align}
By definition, 
$$\mathcal Q_\beta(z; 1- x) = \sum_{k=1}^{3} \frac{\widetilde b_k e^{(1- x)\tau\omega_k}}{1 - z\gamma_k} -\sum_{k=1}^{3} \frac{ \widetilde b_k e^{-(1- x)\tau\omega_k}}{1 - z\gamma_k^{-1}}.$$
Therefore, 
\begin{align}\label{shift_reciprocal_Q}
    \mathcal Q_\beta\left(\frac{1}{z}; x\right) =  z \mathcal Q_\beta(z; 1- x).
\end{align}
Combining \eqref{reciprocal_R} and \eqref{shift_reciprocal_Q}, we obtain
\begin{align*}
\mathcal{P}_\beta\left(\frac{1}{z}; x\right) = \mathcal R_\beta\left(\frac{1}{z}\right)\mathcal Q_\beta\left(\frac{1}{z}; x\right)& = \left(z^{-6}\mathcal R_\beta(z)\right) \left(z \mathcal Q_\beta(z;1- x)\right)\\
&= z^{-5}\mathcal{P}_\beta (z;1- x).
\end{align*}
Multiplying through by $z^{5}$ completes the proof.
\end{pf}

Hence,
\begin{align}\label{reflectio_on_m}
    m_{5,\beta}( x)=m_{0,\beta}(1- x),\quad m_{4,\beta}( x)=m_{1,\beta}(1- x),\qtq{and} m_{3,\beta}( x)=m_{2,\beta}(1- x).
\end{align}
For every $\beta>0$,
$m_{0,\beta}( x)>0$ for $ x\in (0,1]$. 

\begin{rem}\label{Self_reciprocal_P}
For $\beta>0$, $\mathcal{P}_\beta(z;1/2)$ is a self-reciprocal polynomial of degree $5$. 
\end{rem}

 Let $R=e^\tau>1$ and $Y=e^{\tau x}>1$. Then $\sinh(n \tau x) = (Y^n-Y^{-n})/2$ and $\cosh(n \tau) = (R^n+R^{-n})/2$. Therefore,
\begin{align}
    m_{0,\beta}(x)&= \frac{Y^3-Y^{-3}}2 -\left(Y^2-Y^{-2}\right) +\frac{Y-Y^{-1}}2\nonumber\\
    &=\frac{(Y-1)^3(Y+1)(Y^2+1)} {2Y^3}.\nonumber
\end{align}
$\widetilde{\mathrm{A}}_1$ and $\widetilde{\mathrm{A}}_2$ in \eqref{A1_eqn} and \eqref{A2_eqn}, can be written as $\widetilde{\mathrm{A}}_1 = \sum_{k=1}^3 (R^k + R^{-k})$
and
$$R^6 \widetilde{\mathrm{A}}_2 = R^{11}+R^{10}+R^9+R^8+2R^7+3R^6+2R^5+R^4+R^3+R^2+R$$ 
respectively. Let
\begin{align}\label{P1_Y_eqn}
    P_1(Y)=(Y-1)^3(Y+1)(Y^2+1),
\end{align}
$C_1(R)=R^3\, \widetilde{\mathrm{A}}_1$, and $C_2(R)=R^6\, \widetilde{\mathrm{A}}_2$. Then
\begin{align}\label{new_m1_m2}
    m_{1,\beta}(x) = \frac{P_1(RY)-C_1(R)P_1(Y)}{2R^3Y^3}\text{ and }m_{2,\beta}(x) = \frac{P_1(R^2Y)-C_1(R)P_1(RY)+C_2(R)P_1(Y)}{2R^6Y^3}.
\end{align}
From \eqref{new_m1_m2}, we have
\begin{align}\label{m1_m2_at_0}
     m_{1,\beta}(0) = \frac{(R-1)^3(R+1)(R^2+1)}{2R^3}>0, \text{ and } m_{2,\beta}(0) = \frac{ (R-1)^3(R+1)(R^2+1)(R^4+1) }{ 2R^5 }>0
\end{align}
whereas
\begin{align}\label{m1_at1}
m_{1,\beta}(1)=\frac{ (R-1)^3(R+1)(R^2+1)(R^4+1) }{ 2R^5 }>0,
\end{align}
and
\begin{align}\label{m2_at1}
    m_{2,\beta}(1) = - \frac{ (R-1)^3(R+1)(R^2+1)(R^2+R+1) }{ R^4 }<0.
\end{align}

To analyze the rank of the Zibulski-Zeevi matrix, we need precise information on
the signs and zeros of the coefficient functions $m_{1,\beta}$ and
$m_{2,\beta}$. We obtain this through their Wronskians, which will show that
$(m_{0,\beta},m_{1,\beta},m_{2,\beta})$ forms an extended complete Chebyshev system on $(0,1)$.

\begin{lem}\label{Wronskian_Lemma}
    For every $\beta>0$ and $0<x<1$, Wronskian of
$$W(m_{0,\beta},m_{1,\beta})(x)<0, \text{ and }W(m_{0,\beta},m_{1,\beta},m_{2,\beta})(x)<0.$$
\end{lem}

\begin{proof}
 All Wronskians are evaluated algebraically as rational functions of $R=e^\tau>1$ and $Y=e^{\tau x}>1$. By definition,
\begin{align}\label{Wronskian_of_two}
    W(m_{0,\beta},m_{1,\beta})(x)&= m_{0,\beta}(x)m_{1,\beta}'(x)-m_{0,\beta}'(x)m_{1,\beta}(x)\nonumber\\
    &= -\tau \frac{ (R-1)(Y-1)^2(RY-1)^2 }{ 2R^3Y^5 } P_2(R,Y), 
\end{align}
where
$$P_2(R,Y) =R^3Y^6+R^2(R+1)Y^5+R(R^2+R+1)Y^4+3(R+1)(R^2+1)Y^3+(R^2+R+1)Y^2+(R+1)Y+1.$$
Since $\tau>0$, $R>1$, $Y>1$, and $RY>1$, we have
$P_2(R,Y)>0$ and every factor in \ref{Wronskian_of_two} is strictly positive. Therefore, $W(m_{0,\beta},m_{1,\beta})(x)<0$, for $0<x<1$.

Now,
\begin{align*}
    W(m_0,m_1,m_2) &= \tau^3 Y^3 \det \begin{pmatrix} m_0 & m_1 & m_2 \\ \partial_Y m_0 & \partial_Y m_1 & \partial_Y m_2 \\ \partial_Y^2 m_0 & \partial_Y^2 m_1 & \partial_Y^2 m_2 \end{pmatrix}\\
&= -\tau^3 \frac{(R-1)^3(R+1)(Y-1)(RY-1)(RY+1)(R^2Y-1)}{2R^8Y^6} P_3(R,Y),
\end{align*}
where $P_3(R,Y)$ is a real polynomial. Upon setting $R=1+r$ and $Y=1+y$, the polynomial $P_3(R,Y)$ admits an expansion
$$ P_3(1+r,1+y)=\sum_{i,j=0}^8c_{ij}r^iy^j $$
with $c_{ij}>0$. Therefore, $P_3(R, Y)>0$ for $R, Y>1$, and hence the Wronskian is strictly negative. An exact symbolic expansion is reproduced by the \href{https://github.com/Riya74012/Gabor-even-rational-windows-symbolic-verification}{MATLAB} code for verification.
\end{proof}

\begin{rem}\label{m1/m2_decreasing}
 From Lemma \ref{Wronskian_Lemma}, we have
    $$\left(\frac{m_{1,\beta}}{m_{0,\beta}}\right)'=\frac{W(m_{0,\beta},m_{1,\beta}}{m^2_{0,\beta}}<0.$$
    Hence, $\frac{m_1}{m_0}$ is a decreasing function in $[0,1]$. Since $m_{1,\beta}(1)>0$, for $0<x\le 1$, 
    $$\frac{m_{1,\beta}(x)}{m_{0,\beta}(x)}\ge \frac{m_{1,\beta}(1)}{m_{0,\beta}(1)}>0,$$
    and therefore, $m_{1,\beta}(x)>0$.
\end{rem}

Using the ECT property established above, the following Lemma discusses the precise sign behavior of $m_{2,\beta}(x)$, proving it possesses exactly one simple zero.

\begin{lem}\label{unique_Zero_m2}
    $m_{2,\beta}$ has a unique simple zero $\theta_\beta\in(0,1),$ with
$m_{2,\beta}(x)>0$ for $0\le x<\theta_\beta$ and $m_{2,\beta}(x)<0$ for $\theta_\beta<x\le 1$.
\end{lem}

\begin{proof}
From \eqref{m1_m2_at_0} and \eqref{m2_at1}, we have $m_{2,\beta}(0)>0$ and $m_{2,\beta}(1)<0$. Hence, by continuity, $m_{2,\beta}$ has at least one zero in $(0,1)$. On the other hand, from Lemma \eqref{Wronskian_Lemma} we have 
$$ W(m_{0,\beta},m_{1,\beta})(x)<0, \text{ and } W(m_{0,\beta},m_{1,\beta},m_{2,\beta})(x)<0, \quad 0<x<1.$$
Hence,
$$ W(m_{0,\beta},-m_{1,\beta})(x)>0, \qquad W(m_{0,\beta},-m_{1,\beta},m_{2,\beta})(x)>0. $$
Therefore, by the Wronskian characterization of extended complete Chebyshev (ECT) systems \cite[Ch. XI]{KarlinStudden1966}, the system
$$ (m_{0,\beta},-m_{1,\beta},m_{2,\beta}) $$
is an ECT-system on $(0,1)$. Since multiplication of one element by a nonzero constant does not affect the Chebyshev property, $(m_{0,\beta},m_{1,\beta},m_{2,\beta})$ is also an ECT-system. Consequently, every nontrivial linear combination of these three functions has at most two zeros in $(0,1)$, counted with multiplicity. In particular,
$m_{2,\beta}$ itself has at most two zeros counting multiplicities. 

Since
$m_{2,\beta}(0)>0$ and $m_{2,\beta}(1)<0$, it has at least one sign-changing zero $\theta_\beta\in(0,1). $
Such a zero has odd multiplicity. Since the ECT property allows at most two zeros counted with multiplicity, $\theta_\beta$ must be simple.

We claim that it is the only zero. Indeed, if $m_{2,\beta}$ had another distinct zero, then the ECT bound would force this second zero to be simple as well. Hence, both zeros would be sign-changing. Starting with $m_{2,\beta}(0)>0$, two sign changes would imply $ m_{2,\beta}(1)>0$, contrary to $\eqref{m2_at1}$. Therefore $m_{2,\beta}$ has exactly one simple zero $\theta_\beta\in(0,1)$.

Consequently, by continuity,
$m_{2,\beta}(x)>0$ for $0\le x<\theta_\beta$ and $m_{2,\beta}(x)<0$ for $\theta_\beta<x\le 1$.
\end{proof}

\begin{lem}\label{F_xi_lemma}
Let $F_\beta(x)=m_{2,\beta}(x)-m_{1,\beta}(x)$.
\begin{itemize}
    \item[$(i)$] $F_\beta$ has a unique simple zero
$\rho_\beta\in(0,1)$
with $F_\beta(x)>0$ for $0\le x<\rho_\beta$.
    \item[$(ii)$] If $F_\beta(1/3)\le 0$, then $m_{2,\beta}(5/6)<0$.
    \item[$(iii)$] If $F_\beta(1/2)\le0$, then $m_{2,\beta}(3/4)<0$.
\end{itemize}
\end{lem}

\begin{proof}
    From \eqref{m1_m2_at_0}, \eqref{m1_at1}, and \eqref{m2_at1}, we have
    $$F_\beta(0) = \frac{(R-1)^3(R+1)(R^2+1)(R^4-R^2+1)}{2R^5} > 0$$
    and 
    $$F_\beta(1) = -\frac{(R-1)^3(R+1)^3(R^2+1)^2}{2R^5} < 0$$
    for $R=e^\tau>1$.\\
    \noindent
    $(i)$ By continuity, $F_\beta$ has at least one zero in $(0,1)$. By Lemma \ref{Wronskian_Lemma}, $(m_{0,\beta}, m_{1,\beta}, m_{2,\beta})$ forms an ECT system on $(0,1)$. Consequently, any nontrivial linear combination, including $F_\beta = m_{2,\beta} - m_{1,\beta}$, possesses at most two zeros in $(0,1)$ counted with multiplicity. Since $F_\beta$ changes sign from positive to negative, at least one zero $\rho_\beta \in (0,1)$ must have an odd multiplicity. Since the total multiplicity is bounded by $2$, this zero is necessarily simple.
    
    If $F_\beta$ possessed a second distinct zero, the ECT bound would force this second zero to also be simple, creating a second sign change. Starting from $F_\beta(0) > 0$, two sign changes would require $F_\beta(1) > 0$, directly contradicting our boundary evaluation. Thus, $F_\beta$ has exactly one simple zero $\rho_\beta \in (0,1)$. By continuity, $F_\beta(x) > 0$ for $0 \le x < \rho_\beta$ and $F_\beta(x) < 0$ for $\rho_\beta < x \le 1$.\\[0.25cm]
    \noindent
    $(ii)$ Let $s = e^{\tau/6} > 1$. At $x=1/3$, substituting $R = s^6$ and $Y = e^{\tau/3} = s^2$ into $F_\beta(x)$ gives 
    $$F_\beta(1/3) = P_A(s)A(s)$$
    where
    $$P_A(s) = \frac{(s^2-1)(s^{14}-1)(s^{16}-1)}{2s^{32}}> 0$$
    and
    $$A(s) = s^{32}+s^{30}-2s^{28}-3s^{22}+s^{18}-2s^{16}+s^{14}-3s^{10}-2s^4+s^2+1.$$
    Since $P_A(s) > 0$ for $s > 1$, we have $\operatorname{sgn} F_\beta(1/3) = \operatorname{sgn} A(s)$.
    
    Next, consider $m_{2,\beta}(5/6)$. Substituting $R = s^6$ and $Y = e^{5\tau/6} = s^5$ yields
    $$m_{2,\beta}(5/6) = P_B(s)B(s)$$
    where $P_B(s) = \frac{(s-1)^3(s+1)(s^2+1)}{2s^{35}} > 0$. The residual polynomial $B(s)$ is a degree-64 polynomial. The full explicit symbolic expansion and shifted-coefficient verification are provided in the \href{https://github.com/Riya74012/Gabor-even-rational-windows-symbolic-verification}{MATLAB} code. Since $P_B(s) > 0$, we have $\operatorname{sgn} m_{2,\beta}(5/6) = \operatorname{sgn} B(s)$.
    
    We now locate the roots of $A(s)$ and $B(s)$ on $s > 1$ by substituting $s = 1+t$ for $t > 0$. For $A(1+t)$, exact expansion reveals the coefficients of $t^j$ are strictly negative for $0 \le j \le 3$ and strictly positive for $4 \le j \le 32$. This single sign change, combined with $A(1) = -6 < 0$ and $A(s) \to +\infty$, implies by Descartes' rule of signs that $A(s)$ has a unique zero $s_A > 1$. For $B(1+t)$, exact expansion reveals the coefficients of $t^j$ are strictly negative for $0 \le j \le 13$ and strictly positive for $14 \le j \le 64$. With $B(1) = -1198 < 0$, $B(s)$ also possesses a unique zero $s_B > 1$. Evaluating both polynomials exactly at $s = 6/5$ gives
    $$A(6/5)  > 0\text{ and }B(6/5) < 0.$$
    
    Therefore, $1 < s_A < 6/5 < s_B$, establishing the strict ordering $s_A < s_B$. If $F_\beta(1/3) \le 0$, then $A(s) \le 0$, which restricts $s \le s_A$. Since $s_A < s_B$, this forces $s < s_B$, meaning $B(s) < 0$. Consequently, $m_{2,\beta}(5/6) < 0$.\\
    \noindent
    $(iii)$ Let $s = e^{\tau/4} > 1$. At $x=1/2$, substituting $R = s^4$ and $Y = s^2$ gives
    $$F_\beta(1/2) = P_A(s)A(s)$$where $P_A(s) = \frac{(s-1)^3(s+1)^3(s^2+1)(s^4+1)}{2s^{22}} > 0$ and
    $$A(s) = s^{32}-s^{28}-2s^{26}-2s^{24}-4s^{22}-5s^{20}-6s^{18}-5s^{16}-6s^{14}-5s^{12}-4s^{10}-2s^8-2s^6-s^4+1.$$
    At $x=3/4$, substituting $R = s^4$ and $Y = s^3$ gives 
    $$m_{2,\beta}(3/4) = P_B(s)B(s),$$
    where $P_B(s)=\frac{(s-1)^2(s^2+1)(s^{14}-1)}{2s^{23}} > 0$ for $s > 1$, and $B(s)$ is polynomial of degree $28$. An exact symbolic expansion is reproduced by the \href{https://github.com/Riya74012/Gabor-even-rational-windows-symbolic-verification}{MATLAB} code for verification.
    % \begin{align*}
    %     B(s) &= s^{28}-s^{26}+s^{24}-2s^{22}-2s^{21}-s^{20}-2s^{19}-4s^{18}-4s^{17}-3s^{16}-4s^{15}-4s^{14}-4s^{13}-3s^{12}-4s^{11}-4s^{10}-2s^9-s^8-2s^7-2s^6+s^4-s^2+1.
    % \end{align*}
    
    Substituting $s = 1+t$ for $t > 0$, $A(1+t)$ has strictly negative coefficients for $0 \le j \le 7$ and strictly positive for $8 \le j \le 32$. With $A(1) = -43 < 0$, $A$ has a unique zero $s_A > 1$. $B(1+t)$ has strictly negative coefficients for $0 \le j \le 7$ and strictly positive for $8 \le j \le 28$. With $B(1) = -46 < 0$, $B$ has a unique zero $s_B > 1$.Evaluating at the separator $s = 7/5$
    $$A(7/5)> 0\text{ and }B(7/5)< 0.$$
    This establishes the ordering $1 < s_A < 7/5 < s_B$.
    
    Assuming $F_\beta(1/2) \le 0$ forces $A(s) \le 0$, yielding $s \le s_A$. Because $s_A < s_B$, we must have $s < s_B$, which implies $B(s) < 0$. Therefore, $m_{2,\beta}(3/4) < 0$.
    
\end{proof}

The Zak transform of a function $g$ is defined by
\begin{align}\label{Zak}
\mathcal{Z}_g(t, x)=\sum_{m\in\mathbb{Z}}g(t-m)e^{2\pi im x},\quad (t, x)\in \R^2.
\end{align}
\begin{prop}\cite{ghosh2026gabor}\label{P: zak and Q}
If $g$ is defined in \eqref{rational}, then
    $$\mathcal{Z}_{g}(t\,; x)=\frac{\pi}{6}\, e^{2\pi i x t}\,Q(z\,; x),\qtq{for} 0< x<1,$$
    where $z=e^{2\pi it}$ and $Q(z\,; x)$ is defined in \eqref{P_Ndefinition}.
\end{prop}

\begin{rem}\cite{ghosh2026gabor}
At $ x=0$, $\mathcal Z_g(t,0)=\frac{\pi}{6}\, z\, Q(z\,;0).$
\end{rem}

The partial-fraction expansion in \eqref{rational} can be written as
\begin{align}\label{Partial_fraction}
     g(x)
     %&= -\frac1{12}\frac1{x^2+1} +\frac13\frac1{x^2+4} -\frac14\frac1{x^2+9}\\
     &=\sum_{k=1}^3\frac{d_k}{x^2+k^2}, 
\end{align}
where $d_1=-\frac{1}{12}$, $d_2=\frac{1}{3}$, and $d_3=-\frac{1}{4}$. Since the Fourier transform of the standard rational component is given by
$$ \widehat{\frac1{x^2+c^2}}(\xi) = \frac{\pi}{c}e^{-2\pi c| \xi|}, $$
we obtain the Fourier transform of the window function
$$\widehat g(\xi) = -\frac{\pi}{12} \left( e^{-2\pi|\xi|} -2e^{-4\pi|\xi|} +e^{-6\pi|\xi|} \right). $$
Define
\begin{align}\label{h_function}
h_\lambda(\xi) = e^{-\lambda|\xi|} - 2e^{-2\lambda|\xi|} + e^{-3\lambda|\xi|} = e^{-\lambda|\xi|} (1- e^{-\lambda|\xi|})^2.
\end{align}
Fourier duality and dilation give us that the Gabor system $\mathcal{G}(g,\alpha,\beta)$ forms a frame if and only if the modified system $\mathcal{G}(h_\lambda,a,1)$ forms a frame, where
$a=\alpha\beta$ and $\lambda=\frac{2\pi}{\alpha} =\frac{2\pi\beta}{a}.$
Throughout this paper, we use $a$ to denote $\alpha\beta$. Equivalently, define
$g_{\beta}(x)=g(x/\beta)$. Then, the system $\mathcal{G}(g,\alpha,\beta)$ forms a frame if and only if the system $\mathcal{G}(g_\beta, a,1)$ forms a frame. Moreover
$$ \widehat g_{\beta}(\xi) = -\frac{\pi\beta}{12} e^{-2\pi\beta|\xi|} (1-e^{-2\pi\beta|\xi|})^2.$$
In particular, for $\xi > 0$ and $\beta > 0$,
$\widehat {g_{\beta}}(\xi)\neq 0.$

Let us define $g_\beta(t)=g(t/\beta)$.
Let $\mathbf M(t\,; x)$ be a $q\times p$ Zibulski-Zeevi matrix \cite{zibulski1997analysis} for $g_\beta$ with entries
\begin{align}\label{M_matrix}
    \mathbf M_{sr}(t\,; x)=\mathcal Z_{g_\beta}(t-a s; x+r/p)
\end{align}
for $(t\,; x)\in [0,1)\times[0,1/p)$. We define an associated $q \times p$ matrix $\mathbf B(z; x)$ on $(z\,; x)\in \mathbb{T}\times[0,1/p)$ with entries 
$$(\mathbf B (z; x))_{sr} = e^{-2\pi i\frac{rs}{q}}Q_\beta\left(ze^{-2\pi ia s}; x+\frac{r}{p}\right),$$ 
where
$$Q_\beta(z; x)=\sum_{k=1}^6 \frac{\widetilde b_k e^{\tau x \omega_k}}{1-ze^{\tau \omega_k}}$$
for $0\le s\le q-1$ and $0\le r\le p-1$.
From Proposition \ref{P: zak and Q}, 
$$\mathbf M_{sr}(t\,; x)=
\Bigl(
\frac{\pi\beta}{6}\,
e^{2\pi i x(t-a s)}
e^{2\pi i(r/p)t}
\Bigr)
\,(\mathbf B (z\,; x))_{sr},\quad z=e^{2\pi it}.
$$
Therefore,
$$\mathbf M(t\,; x)=
\operatorname{diag}(\mathbf D_s(t\,; x))\,
\mathbf{B}(z\,; x)\,
\operatorname{diag}(E_r(t)),
$$
where
$$\mathbf D_s(t\,; x)=\frac{\pi}{6}\,\beta e^{2\pi i x(t-a s)}
\qtq{and} E_r(t)=
e^{2\pi i(r/p)t}.$$
Since both diagonal matrices are invertible 
$$\rank \mathbf{M}(t\,; x)=\rank \mathbf{B}(z\,; x).$$
% Now define
% $$R_\beta(z)=\prod_{k=1}^{3}
% (1-ze^{2\pi\beta k})
% (1-ze^{-2\pi\beta k}),$$
% and
% $$\mathcal{P}_\beta(z; x)=
% \mathcal R_\beta(z) \mathcal Q_\beta(z; x)=\sum_{l=0}^{5}m_{l,\beta}( x)z^l.$$

Finally, we introduce a $q\times p$ matrix $ \Phi(z\,;  x)$
\begin{align}\label{final Phi_alpha matrix}
    &= \begin{bmatrix}
    \mathcal{P}_\beta(z\,; x)& \mathcal{P}_\beta\left(z\,; x+\frac{1}{p}\right)& \cdots& \mathcal{P}_\beta\left(z\,; x+\frac{p-1}{p}\right)\\[0.25em]
   \mathcal{P}_\beta(ze^{-2\pi ia}\,; x)& e^{-\frac{2\pi i}{q}}\mathcal{P}_\beta\left(ze^{-2\pi ia}\,; x+\frac{1}{p}\right)& \cdots&e^{-2\pi i\frac{p-1}{q}}\mathcal{P}_\beta\left(ze^{-2\pi ia}\,; x+\frac{p-1}{p}\right)\\[0.25em]
    \vdots&\vdots&\cdots&\vdots\\[0.25em]
    \mathcal{P}_\beta(ze^{-2\pi ia(q-1)}\,; x)& e^{\frac{2\pi i}{q}}\mathcal{P}_\beta\left(ze^{-2\pi ia(q-1)}\,; x+\frac{1}{p}\right)& \cdots&e^{2\pi i\frac{p-1}{q}}\mathcal{P}_\beta\left(ze^{-2\pi ia(q-1)}\,; x+\frac{p-1}{p}\right)
\end{bmatrix}.
\end{align}
Consequently, from \eqref{P_Ndefinition}, we have
\begin{align*}
(\Phi(z\,; x))_{sr}&=e^{-2 \pi i \frac{r}{q}s}\,\mathcal{P}_\beta\left(ze^{-2 \pi ia s}\, ; x+\frac{r}{p}\right),\\
&=e^{-2\pi i\frac{r}{q}s}
\mathcal R_\beta(ze^{-2\pi ia s})\,\mathcal Q_\beta\left(
ze^{-2\pi ia s}\,;
 x+\frac r p
\right) 
\end{align*}
for $z=e^{2\pi it}$. Since the zeros of $\mathcal R_\beta$ are $e^{\pm2\pi\beta k}$ in \eqref{R_beta}, none lie on $\mathbb T$. Hence, for $|z|=1$,
$$\mathcal R_\beta(z_s)\neq0,\quad z_s=ze^{-2\pi ia s}$$
so
$$\operatorname{rank}\mathbf{M}(t, x)=\operatorname{rank}
\Phi(z; x).$$

The following theorem establishes that the frame property for the Gabor system generated by $g$ is equivalent to the full-rank condition of the Zibulski--Zeevi matrix \eqref{M_matrix} and its equivalent matrix representations \eqref{final Phi_alpha matrix}.
\begin{theorem}\cite{ghosh2026gabor}\label{rank equiv}
    Let $g$ be a window function in \eqref{window}. Then the Gabor system $\mathcal{G}(g,\alpha,\beta)$ forms a frame for $\alpha\beta=\frac{p}{q}$ with $\gcd (p,q)=1$ if and only if 
    $$\rank \mathbf{M}(t\,; x)=\rank \mathbf{B} (z\,; x)=\rank \Phi(z\,; x)=p,$$
    for $(z\,, x)\in \mathbb{T}\times \left[0,1/p\right)$.
\end{theorem}

To explicitly verify the full-rank condition of Theorem \ref{rank equiv}, we must analyze the algebraic properties of its associated matrix $\Phi(z\,; x)$. By leveraging the polynomial reflection identities in \eqref{eq:poly_reciprocity}, the following lemma demonstrates a fundamental structural symmetry in the frequency parameter $ x$. Consequently, it is sufficient to analyze the $\Phi(z\,; x)$ matrix only for $ x\in [0, \frac{1}{2p}]$.

\begin{lem}\cite{ghosh2026gabor}\label{symmetry_singular_values}
Let $\alpha\beta=p/q$ and let $\Phi(z; x)$ be the matrix associated
with the window $g$. Then, for every $z\in\mathbb T$ and
$ x\in[0,1/p)$, there exist permutation matrices $R$ and $P$ and a diagonal unitary matrix $U(z)$ such that
\begin{equation}\label{Phi_reflection}
\Phi\left(z;\frac1p- x\right)
=
U(z)\,R\,\Phi(z^{-1}; x)\,P .
\end{equation}
Consequently,
$$
\operatorname{rank}\Phi\left(z;\frac1p- x\right)
=
\operatorname{rank}\Phi(z^{-1}; x).
$$
In particular, it is enough to verify the full-rank condition for $ x\in\left[0,\frac1{2p}\right].$
\end{lem}

% This dilation issue is important: after normalization, the poles are
% $$ \pm i\beta,\quad \pm2i\beta,\quad\pm3i\beta, $$
% not the original $\pm i,\pm2i,\pm3i$.

\section{Frame region for \texorpdfstring{$g$}{g2}} \label{Section_3}
\begin{proof}[\textbf{Proof of Theorem \ref{frame_region0}}]
We now prove the frame property of
$\mathcal G(h_\lambda,a,1)$. For $\xi\in\mathbb R$, consider its Ron--Shen \cite{ronshen} pre-Gramian
$$ P_{\xi}= \bigl( h_\lambda(\xi+aj-k) \bigr)_{j,k\in\mathbb Z}.$$
It is enough to prove that there exists a constant $C>0$, independent of $\xi$, such that
$$ \|P_{\xi} c\|_{\ell^2} \ge C\|c\|_{\ell^2}, \quad c\in\ell^2(\mathbb Z).$$
A direct computation gives $h_\lambda(0)=h_\lambda'(0)=0$, and the left and right second derivatives agree. Hence $h_\lambda\in C^2(\mathbb{R})$. Let 
\begin{align}\label{H_c_func}
    H_c(\xi)=\sum_{k\in \Z} c_k h_\lambda(\xi-k).
\end{align}
Fix $m\in\mathbb Z$ and let $m<\xi<m+1$. Then, for $k\le m$, we have
$|\xi-k|=\xi-k$, whereas for $k\ge m+1$, we have $|\xi-k|=k-\xi$. Hence, \eqref{H_c_func} becomes
\begin{align*}
H_c(\xi)&=\sum_{k\le m}c_k h_\lambda(\xi-k)
  +\sum_{k\ge m+1}c_k h_\lambda(\xi-k)\\
&=\sum_{r=1}^3 a_r e^{-r\lambda \xi}
 \sum_{k\le m} c_k e^{r\lambda k}+\sum_{r=1}^3 a_r e^{r\lambda \xi} \sum_{k\ge m+1} c_k e^{-r\lambda k},
\end{align*}
where $(a_1,a_2,a_3)=(1,-2,1)$.
Therefore,
\[H_c(\xi)=A_{1,m}e^{-\lambda \xi}
-2A_{2,m}e^{-2\lambda \xi}
+A_{3,m}e^{-3\lambda \xi}+B_{1,m}e^{\lambda \xi}
-2B_{2,m}e^{2\lambda \xi}
+B_{3,m}e^{3\lambda \xi},\]
where
\[
A_{r,m}=\sum_{k\le m}c_k e^{r\lambda k},
\qtq{and}
B_{r,m}=\sum_{k\ge m+1}c_k e^{-r\lambda k}.
\]
Thus,
\[
H_c|_{(m,m+1)}
\in
E_\lambda:=\operatorname{span}\{e^{\pm\lambda\xi},e^{\pm2\lambda \xi},e^{\pm3\lambda \xi}\}.
\]
Since the six exponents are distinct and real, $E_\lambda$ is an extended complete Chebyshev space on every bounded interval.
Now fix $L\in\mathbb{N}$ and write $I_m=(m,m+L)$. Consider the spline space
$$\mathcal{S}_{\lambda,L}(I_m)=\left\{f\in C^2(I_m): f\vert{}_{(m+r,m+r+1)}\in E_\lambda,\;0\le r\le L-1\right\}.$$
Each of the $L$ cells contributes six coefficients. At each of the $L-1$ interior integer knots, continuity of $f,f',f''$ gives three independent matching conditions. Consequently,
$$\dim\mathcal{S}_{\lambda,L}=6L-3(L-1)=3L+3.$$
In particular, $H_c\vert{}_{I_m}\in\mathcal{S}_{\lambda,L}(I_m)$.

Since $1/a-3>0$, choose $L$ sufficiently large so that $L(1/a-3)>3$, which is equivalently 
\begin{align}\label{L/a>3L+3}
    \frac{L}{a}>3L+3.
\end{align}
For a phase $\theta\in\mathbb{R}/a\mathbb{Z}$, consider the lattice $X_\theta=\theta+a\mathbb{Z}$. An open interval of length \(s\) contains at least \(\lceil s/a\rceil-1\) points of \(\theta+a\mathbb Z\). Since $1/a>3$, every open interval of length one contains at least three points of $X_\theta$. Likewise, by \eqref{L/a>3L+3}, every interval of length $L$ contains at least $3L+3$ lattice points.

In each cell $(m+r,m+r+1)$ for $r=0,\ldots,L-1$, choose three points of $X_\theta$. This gives $3L$ points. Choose three further lattice points from $I_m$, and arrange the resulting $3L+3$ nodes 
$t_1<t_2<\cdots<t_{3L+3}$. For each $j=1,\dots, L-1$, the first $j$ cells $(m,m+1),\dots,(m+j-1,m+j)$
contain the $3j$ points chosen there, so at least $3j$ selected nodes lie to the left of the knot $m+j$. Therefore, $t_{3j}<m+j$.
Moreover, besides those $3j$ points, at most the three additional nodes may also lie to the left of $m+j$. Hence, at most $3j+3$ selected nodes lie at or to the left of $m+j$, which gives $m+j<t_{3j+4}$. Thus,
\begin{align}\label{Schoenberg-Whitney_inequality}
    t_{3j}<m+j<t_{3j+4},
\quad j=1,\dots,L-1.
\end{align}
To apply the generalized Schoenberg--Whitney theorem \cite{karlinziegler, schoenbergwhitney, schumaker}, note that \(\mathcal S_{\lambda,L}(I_m)\) is a Chebyshevian spline space of order \(6\). Its extended knot sequence consists of six copies of each endpoint and three copies of every interior knot
$$ \underbrace{m,\ldots,m}_{6}, \underbrace{m+1,\ldots,m+1}_{3}, \ldots, \underbrace{m+L-1,\ldots,m+L-1}_{3}, \underbrace{m+L,\ldots,m+L}_{6}.$$
If this sequence is denoted by \((\tau_i)\), the generalized Schoenberg--Whitney conditions are
\begin{align}\label{knots_SW}
     \tau_i<t_i<\tau_{i+6}, \quad i=1,\ldots,3L+3.
\end{align}
Because all the nodes lie in \((m,m+L)\), the conditions involving the endpoints are automatic, while at the interior knots \eqref{knots_SW} reduces precisely to \eqref{Schoenberg-Whitney_inequality}.

By the generalized Schoenberg--Whitney theorem, the collocation map
\[T_\theta:\mathcal S_{\lambda,L}(I_m)\longrightarrow\mathbb C^{3L+3},\quad T_\theta f=\bigl(f(t_1),\ldots,f(t_{3L+3})\bigr),
\]
is injective. Since $\dim\mathcal S_{\lambda,L}(I_m)=3L+3$,
the map \(T_\theta\) is an isomorphism. Consequently, for each fixed
phase \(\theta\), there exists \(C_\theta>0\) such that
\begin{align}\label{collocation_isomorphism}
\|f\|_{L^2(I_m)}^2\le C_\theta\sum_{i=1}^{3L+3}|f(t_i)|^2,\quad f\in\mathcal S_{\lambda,L}(I_m).
\end{align}

We next show that the constant in \eqref{collocation_isomorphism} can be chosen
independently of the phase \(\theta\). It is enough first to consider \(m=0\), since integer translation maps
\(\mathcal S_{\lambda,L}(I_m)\) isomorphically onto
\(\mathcal S_{\lambda,L}((0,L))\). Fix a phase
$\theta_0\in\mathbb R/a\mathbb Z$.
Choose \(3L+3\) lattice points
\[
t_i(\theta_0)=\theta_0+a n_i,
\quad i=1,\ldots,3L+3,
\]
satisfying the Schoenberg--Whitney inequalities \eqref{Schoenberg-Whitney_inequality}. Here the
integers \(n_1<\cdots<n_{3L+3}\) are fixed. Since all inequalities in \eqref{Schoenberg-Whitney_inequality} are strict, there exists a
neighborhood \(U_{\theta_0}\) of \(\theta_0\) such that, for every
\(\theta\in U_{\theta_0}\), the same lattice indices \(n_i\) give nodes
\[
t_i(\theta)=\theta+a n_i
\]
which still lie in \((0,L)\) and satisfy the same
Schoenberg--Whitney inequalities.

Choose a basis
$\varphi_1,\ldots,\varphi_{3L+3}\in \mathcal S_{\lambda,L}((0,L))$, and let
\[
M(\theta)
=
\bigl(\varphi_\ell(t_i(\theta))\bigr)_
{1\le i,\ell\le3L+3}
\]
be the corresponding collocation matrix. Since each
\(\varphi_\ell\) is continuous, \(M(\theta)\) depends continuously on
\(\theta\). By the Schoenberg--Whitney theorem, $\det M(\theta_0)\neq 0$.
Hence, after possibly shrinking \(U_{\theta_0}\),
\[
\inf_{\theta\in U_{\theta_0}}
\sigma_{\min}(M(\theta))>0.
\]

Therefore, there exists a constant \(C_{\theta_0}<\infty\) such that
\[\|f\|_{L^2(0,L)}^2\le C_{\theta_0}\sum_{i=1}^{3L+3} |f(t_i(\theta))|^2
\]
for every \(\theta\in U_{\theta_0}\) and every
\(f\in\mathcal S_{\lambda,L}((0,L))\). The space $\mathbb R/a\mathbb Z$ is compact. Hence, finitely many neighborhoods $U_{\theta_1},\ldots,U_{\theta_N}$of the above type cover it. Taking
\[C_{a,\lambda,L}=\max_{1\le \nu\le N}C_{\theta_\nu},
\]
we obtain a constant independent of \(\theta\) such that
\[\|f\|_{L^2(0,L)}^2\le C_{a,\lambda,L}
\sum_{\substack{y\in X_\theta\\0<y<L}}
|f(y)|^2.\]
Here we have enlarged the sum from the selected \(3L+3\) nodes to all lattice points in the interval, which can only increase its right-hand side. By integer translation, the same constant works on every interval
\(I_m=(m,m+L)\). Thus
\[\|f\|_{L^2(m,m+L)}^2
\le
C_{a,\lambda,L}
\sum_{\substack{y\in X_\theta\\m<y<m+L}}|f(y)|^2,\]
uniformly in \(m\in\mathbb Z\), in the phase
\(\theta\in\mathbb R/a\mathbb Z\), and in
\(f\in\mathcal S_{\lambda,L}(I_m)\).

Almost every point $\xi\in\mathbb{R}$ belongs to exactly $L$ intervals of the form $(m,m+L)$. Consequently,
$$\sum_{m\in\mathbb{Z}}\Vert{}H_c\Vert{}_{L^2(m,m+L)}^2=L\Vert{}H_c\Vert{}_2^2.$$
Similarly, every point of $x+a\mathbb{Z}$ belongs to at most $L$ of these intervals. Hence,
$$L\Vert{}H_c\Vert{}_2^2\le C_{a,\lambda,L}\,L\sum_{j\in\mathbb{Z}}\vert{}H_c(\xi+aj)\vert{}^2.$$
Therefore,
\begin{align}\label{sum_Fc_bound}
\sum_{j\in\mathbb{Z}}\vert{}H_c(\xi+aj)\vert{}^2\ge C_{a,\lambda,L}^{-1}\Vert{}H_c\Vert{}_2^2,
\end{align}
uniformly in $x$. With the Fourier transform
% $$\widehat{e^{-\gamma\vert{}\cdot\vert{}}}(\xi)=\frac{2\gamma}{\gamma^2+4\pi^2\xi^2}.$$
% Hence,
\begin{align*}
\widehat{h}_\lambda(\omega)=&\frac{2\lambda}{\lambda^2+4\pi^2\omega^2}-\frac{8\lambda}{4\lambda^2+4\pi^2\omega^2}+\frac{6\lambda}{9\lambda^2+4\pi^2\omega^2}\\
=&\frac{24\lambda^3(\lambda^2-4\pi^2\omega^2)}{(\lambda^2+4\pi^2\omega^2)(4\lambda^2+4\pi^2\omega^2)(9\lambda^2+4\pi^2\omega^2)}=-\frac{24}{\lambda}g\left(\frac{2\pi}{\lambda}\omega\right).
\end{align*}
Thus, the only real zeros of $\widehat{h}_\lambda$ are $\omega=\pm\frac{\lambda}{2\pi}.$
Define
\begin{align}\label{periodized_Fourier_Gramian}
G_\lambda(\omega)=\sum_{n\in\mathbb{Z}}\vert{}\widehat{h}_\lambda(\omega+n)\vert{}^2.
\end{align}
Since $\widehat{h}_\lambda(\omega)=O(\vert{}\omega\vert{}^{-4})$ as $\vert{}\omega\vert{}\to\infty$, the series in \eqref{periodized_Fourier_Gramian} converges uniformly for $\omega\in[0,1]$. Hence, $G_\lambda$ is continuous and 1-periodic.

Moreover,
$G_\lambda(\omega)>0$ for every $\omega\in\mathbb{R}.$ Indeed, if $G_\lambda(\omega)=0$, then $\widehat{h}_\lambda(\omega+n)=0$ for all $n\in\mathbb{Z}$. This would force the infinite set $\omega+\mathbb{Z}$ to be contained in the two-point set $\{-\frac{\lambda}{2\pi},\frac{\lambda}{2\pi}\}$, which is impossible. Therefore, by compactness,
$$A_\lambda:=\min_{\omega\in[0,1]}G_\lambda(\omega)>0.$$
Let $C(\omega)=\sum_{k\in\mathbb{Z}}c_ke^{-2\pi ik\omega}$. Since
$\widehat H_c(\omega) = \widehat h_\lambda(\omega)C(\omega)$, 
Parseval's identity and periodization give
\begin{align}\label{H_c_eqn}
\Vert{}H_c\Vert{}_2^2&=\int_0^1\vert{}C(\xi)\vert{}^2G_\lambda(\xi)\,d\xi\nonumber\\
    &\ge A_\lambda\int_0^1\vert{}C(\xi)\vert{}^2\,d\xi=A_\lambda\Vert{}c\Vert{}_{\ell^2}^2.
\end{align}
Combining \eqref{sum_Fc_bound} and \eqref{H_c_eqn}, we obtain
$$\sum_{j\in\mathbb{Z}}\left\vert{}\sum_{k\in\mathbb{Z}}c_kh_\lambda(\xi+aj-k)\right\vert{}^2\ge\frac{A_\lambda}{C_{a,\lambda,L}}\Vert{}c\Vert{}_{\ell^2}^2,$$
uniformly in $\xi$. Since $h_\lambda$ decays exponentially, the corresponding upper pre-Gramian estimate is automatic. The Ron--Shen criterion therefore implies that $\mathcal{G}(h_\lambda,a,1)$ is a frame for $L^2(\mathbb{R})$. Hence, we prove our result.
\end{proof}

\begin{proof}[\textbf{Proof of Theorem \ref{frame_region1}}]
    We now prove the frame property of
$\mathcal G(h_\lambda,a,1)$. For $\xi\in\mathbb R$, consider its Ron--Shen \cite{ronshen} pre-Gramian
$$ P_{\xi}= \bigl( h_\lambda(\xi+aj-k) \bigr)_{j,k\in\mathbb Z}.$$
It is enough to prove that there exists a constant $C>0$, independent of $\xi$, such that
$$ \|P_{\xi} c\|_{\ell^2} \ge C\|c\|_{\ell^2}, \quad c\in\ell^2(\mathbb Z).$$
Since $a<a_*<1/2$, throughout the proof we have $ 0<a<\frac12$. For each \(k\in\mathbb Z\), choose \(n\in\mathbb Z\) such that
$$ \xi+an\le k\le \xi+a(n+1). $$
The two consecutive lattice points \(\xi+an\) and \(\xi+a(n+1)\) are separated by distance \(a\), and the sum of their distances from \(k\) is therefore \(a\). Choosing $y_k=\xi+aj_k$, for some $j_k\in\mathbb Z$ to be the farther of the two gives
$$ \frac a2\le |y_k-k|\le a. $$
The indices $j_k$ are all distinct. Indeed, if $y_k=y_l$ for $k\ne l$, then
$$ |k-l| \le |k-y_k|+|y_l-l| \le2a<1, $$
which is impossible because $k-l$ is a nonzero integer. Hence, the rows indexed by $j_k$ form a legitimate row-submatrix of $P_x$. Denote this submatrix by
$$A_{\xi}=\left[h_\lambda(y_k-l)\right]_{k,l\in\Z}.$$
Since
$\frac{\lambda a}{2} = \pi\beta\ge\pi>\log3, $
the function
$$ h_\lambda(\xi) = e^{-\lambda \xi}(1-e^{-\lambda \xi})^2, \quad \xi>0, $$
is decreasing on $[a/2,\infty)$. Therefore, $h_\lambda(y_k-k) \ge h_\lambda(a).$
Put $\rho=e^{-\lambda a}=e^{-2\pi\beta}.$
Then every selected diagonal entry satisfies
$$h_\lambda(y_k-k)\ge\rho(1-\rho)^2.$$
Write
$A_{\xi}=D_{\xi}+E_{\xi}$, where $D_{\xi}$ is  a diagonal matrix with
$(D_{\xi})_{kk}=h_\lambda(y_k-k)$
and
$\|D_{\xi} c\|_{\ell^2} \ge \rho(1-\rho)^2\|c\|_{\ell^2}$.
We now estimate the off-diagonal part. For $l\neq k$,
$|y_k-l|\ge |k-l|-a.$ Since
$0<h_\lambda(t)\le e^{-\lambda|t|},$
the off-diagonal row sum is bounded by
\begin{align} 
\sum_{l\neq k}|h_\lambda(y_k-l)| &\le 2\sum_{n=1}^\infty e^{-\lambda(n-a)}= \frac{2e^{-\lambda(1-a)}}{1-e^{-\lambda}}. \nonumber
\end{align}
The same estimate holds for column sums. Hence, the Schur test gives
$$\|E_{\xi}\|_{\ell^2\to\ell^2}\le\frac{2e^{-\lambda(1-a)}}{1-e^{-\lambda}}.$$
Consequently,
$$\|A_{\xi} c\|_{\ell^2}\ge \|D_{\xi} c\|_{\ell^2}-\|E_{\xi} c\|_{\ell^2}\ge \left[\rho(1-\rho)^2-\frac{2e^{-\lambda(1-a)}}{1-e^{-\lambda}}\right]\|c\|_{\ell^2}.$$
Thus the selected Ron--Shen submatrix is bounded below whenever
$$\frac{2e^{-\lambda(1-a)}}{1-e^{-\lambda}}<\rho(1-\rho)^2. $$
Since $ e^{-\lambda}=\rho^{1/a}$,
this condition is equivalent to
$$f(a,\rho)=\frac{ 2\rho^{1/a-2} }{ (1-\rho^{1/a})(1-\rho)^2 }<1.$$
For fixed $a<1/2$, $f(a,\rho)$ is also strictly increasing in $\rho\in (0,1)$ for $\beta\ge 1$. Therefore, $f(a,\rho)\le f(a, e^{-2\pi})$. Moreover, $f(a,e^{-2\pi})$ is continuous and strictly increasing in $a\in(1/3,1/2)$ with
$$f\left(\frac{1}{3},e^{-2\pi}\right)<1,\qtq{and} \lim\limits_{a\rightarrow 1/2}f(a,e^{-2\pi})=\frac{2}{(1-e^{-4\pi})(1-e^{-2\pi})^2}>1.$$
Hence, there exists a unique $a_*\in (1/3,1/2)$ such that
$$f(a,\rho)\le f(a,e^{-2\pi})< f(a_*,e^{-2\pi})=1.$$ Numerically, $a_*\approx 0.47373.$
Hence,
$$ \delta:= \rho(1-\rho)^2- \frac{2e^{-\lambda(1-a)}}{1-e^{-\lambda}}>0, $$
and therefore
$\|A_\xi c\|_2\ge\delta\|c\|_2. $
Since \(A_\xi\) is obtained by selecting rows of \(P_\xi\),
$$ \|P_\xi c\|_2\ge\|A_\xi c\|_2\ge\delta\|c\|_2, $$
uniformly in \(\xi\).
Thus, we prove our result.
\end{proof}

\begin{pf}[\textbf{Proof of Theorem \ref{frame_region2}}]
Let
$q=3p-1$ and $ x_r= x+\frac rp$, for $0\le x\le\frac1{2p}$ and $0\le r\le p-1$. Suppose
$c=(c_0,\ldots,c_{p-1})\in\ker\Phi.$
Then the kernel equations are
$$ D_k(z\,; x)= \sum_{\substack{0\le r<p,\;0\le l\le5\\ r+pl\equiv k\;(\mathrm{mod}\ q)}} c_rm_{l,\beta}( x_r)z^l=0.$$
\textbf{Case 1: For $\bm p\ge3$.} For $0\le r\le p-2$, we have
$$ c_rm_{4,\beta}( x_r)z^4 + c_{r+1}m_{1,\beta}( x_{r+1})z=0. $$
Since $m_{1,\beta},m_{4,\beta}>0$, hence
$$c_{r+1} = -z^3c_r \frac{m_{4,\beta}( x_r)}{m_{1,\beta}( x_{r+1})}.$$
Thus, a nonzero kernel vector would have every $c_r\neq0$. The two boundary binomials are
$$ c_0m_{1,\beta}( x_0)z+ c_{p-1}m_{3,\beta}( x_{p-1})z^3=0$$
and
\begin{align}\label{case1_thm3.2_eqn1}
     c_0m_{2,\beta}( x_0)z^2+ c_{p-1}m_{4,\beta}( x_{p-1})z^4=0.
\end{align}
Let $y_0=1- x_{p-1} = \frac1p- x.$
Eliminating $c_{p-1}/c_0$ and using reflection in \eqref{reflectio_on_m} gives
\begin{align}\label{thm3.2_eqn1}
    m_{2,\beta}( x_0)m_{2,\beta}(y_0) = m_{1,\beta}( x_0)m_{1,\beta}(y_0)>0.
\end{align}
Hence, $m_{2,\beta}( x_0)$ and $m_{2,\beta}(y_0)$ have the same sign. If $m_{2,\beta}( x_0)<0$ and $m_{2,\beta}(y_0)<0$, then $ \theta_\beta< x_0$. Since $p\ge 3$, we have
$$D_{2p+1}(z\,; x)=c_0m_{5,\beta}( x_0)z^5+c_1m_{2,\beta}( x_1)z^2=0.$$
Using $c_1=-a_0c_0 z^3$ with $a_0=\frac{m_{4,\beta}(x_0)}{m_{1,\beta}(x_{1})}>0$, we obtain
$m_{2,\beta}( x_1) = \frac{m_{5,\beta}( x_0)}{a_0}>0$ which is a contradiction because $ x_1> x_0>\theta_\beta$. Therefore, we must have $m_{2,\beta}( x_0)>0$ and $m_{2,\beta}(y_0)>0$.

Now, we define
$$b_0=\frac{m_{2,\beta}(x_0)}{m_{4,\beta}(x_{p-1})}=\frac{m_{2,\beta}(x_0)}{m_{1,\beta}(y_0)}>0.$$
From \ref{case1_thm3.2_eqn1}, we have $c_{p-1}=-b_0 c_0 z^{-2}$. For $k=1$, we have
$$D_1(z\,; x)=c_0m_{3,\beta}( x_0)z^3+c_1m_{0,\beta}( x_1)+c_{p-1}m_{5,\beta}( x_{p-1})z^5=0.$$
It becomes
$$ m_{3,\beta}( x_0)-a_0m_{0,\beta}( x_1) -b_0 m_{5,\beta}( x_{p-1})=0.$$
Since $a_0, b_0>0$, we have $m_{3,\beta}( x_0)>0$. By reflection, $m_{2,\beta}(1- x_0)>0$. Since
$\theta_\beta>1- x_0 \ge 1-\frac1{2p} \ge\frac56$, we have
\begin{align}\label{thm3.2_case1}
    m_{2,\beta}(5/6)>0.
\end{align}
On the other hand, \eqref{thm3.2_eqn1} gives
$$ \frac{m_{2,\beta}( x_0)}{m_{1,\beta}( x_0)} \frac{m_{2,\beta}(y_0)}{m_{1,\beta}(y_0)} =1. $$
Consequently, one of the two ratios is at most $1$, while the other is at least $1$. Hence, $F_\beta( x_0)F_\beta(y_0)\le 0$. Thus, the unique zero $\rho_\beta$ of
$F_\beta( x)$ lies between $ x_0$ and $y_0$. Since $0\le  x_0,y_0\le\frac1p\le\frac13,$ we obtain $F_\beta(1/3)\le 0$.
From Lemma \ref{F_xi_lemma}, we obtain $m_{2,\beta}(5/6)<0$ which is a contradiction with \eqref{thm3.2_case1}. Hence, $\ker\Phi(z\,; x)=\{0\}$.
\\[0.25cm]
\noindent
\textbf{Case 2: For \bm{$p=2$}.}
Here $a=\frac25$ and $ x\in [0,\frac14]$. Let $x= x$ and $y=\frac12- x$.
The five equations are
\begin{align}
c_0m_{0,\beta}(x) +c_1m_{2,\beta}\left(x+\tfrac{1}{2}\right)z^2 +c_0m_{5,\beta}(x)z^5 &= 0, \label{1a} \\
c_1m_{0,\beta}\left(x+\tfrac{1}{2}\right) +c_0m_{3,\beta}(x)z^3 +c_1m_{5,\beta}\left(x+\tfrac{1}{2}\right)z^5 &= 0, \label{1b} \\
c_0m_{1,\beta}(x)z+c_1m_{2,\beta}(y)z^3 &= 0, \label{1c} \\
c_1m_{1,\beta}\left(x+\tfrac{1}{2}\right)z +c_0m_{1,\beta}(1-x)z^4 &= 0, \label{1d} \\
c_0m_{2,\beta}(x)z^2+c_1m_{1,\beta}(y)z^4 &= 0. \label{1e}
\end{align}
Since $m_{1,\beta}>0$, a nontrivial solution requires $c_0c_1\neq0$. Eliminating $c_1/c_0$ from \eqref{1c}, \eqref{1d}, and \eqref{1e} yields the magnitude condition
\begin{equation}
m_{1,\beta}(x)m_{1,\beta}(y) = m_{2,\beta}(x)m_{2,\beta}(y) \label{eqn_2_case2}
\end{equation}
and the phase condition
\begin{equation}
z^5 = \frac{ m_{1,\beta}(x) m_{1,\beta}(x+\frac{1}{2}) }{ m_{2,\beta}(y)m_{1,\beta}(1-x) }. \label{eqn_3_case2}
\end{equation}
Because the right-hand side of \eqref{eqn_3_case2} is real and $|z|=1$, we must have $z^5=\pm1$.

If $z^5=-1$, then \eqref{eqn_3_case2} forces $m_{2,\beta}(y)<0$, which by \eqref{eqn_2_case2} requires $m_{2,\beta}(x)<0$. By Lemma \ref{unique_Zero_m2}, $m_{2,\beta}$ has a unique zero and remains negative to its right, so $m_{2,\beta}(x+\frac{1}{2})<0$. Using \eqref{1c} to eliminate $c_1/c_0$ from \eqref{1a}, we obtain
\begin{equation}\label{eqn_4_case2}
m_{0,\beta}(x)-m_{0,\beta}(1-x) - \frac{m_{1,\beta}(x)} {m_{2,\beta}(y)} m_{2,\beta}\left(x+\tfrac{1}{2}\right) = 0. 
\end{equation}
Since $0\le x\le 1/4$, we have $x < 1-x$. Because $m_{0,\beta}$ is strictly increasing, $m_{0,\beta}(x)-m_{0,\beta}(1-x)<0$. Furthermore, the product of the remaining terms in \eqref{eqn_4_case2} is strictly negative. Thus, the left-hand side is strictly negative, reaching a contradiction. Therefore, we must have $z^5=1$.

With $z^5=1$, \eqref{eqn_2_case2} and \eqref{eqn_3_case2} force $m_{2,\beta}(x)>0$ and $m_{2,\beta}(y)>0$. Isolating $c_1/c_0$ in \eqref{1d} and substituting into \eqref{1b} yields
$$m_{3,\beta}(x) = \frac{m_{1,\beta}(1-x)} {m_{1,\beta}(x+\frac{1}{2})} \left[ m_{0,\beta}\left(x+\tfrac{1}{2}\right) + m_{5,\beta}\left(x+\tfrac{1}{2}\right) \right] > 0.$$
Applying the reflection identity $m_{3,\beta}(x)=m_{2,\beta}(1-x)$, we obtain $m_{2,\beta}(1-x)>0$. Since $1-x \ge 3/4$ and $m_{2,\beta}(1)<0$, the unique zero of $m_{2,\beta}$ must lie to the right of $1-x$. Consequently, 
\begin{align}\label{eqn_5_case2}
    m_{2,\beta}\left(\tfrac{3}{4}\right)>0.
\end{align}

From \eqref{eqn_2_case2}, we have
$$ \frac{m_{2,\beta}(x)}{m_{1,\beta}(x)} \frac{m_{2,\beta}(y)}{m_{1,\beta}(y)} =1. $$
Since both factors are positive, one is at most 1 and the other at least 1. Since $F_\beta$ in Lemma \eqref{F_xi_lemma} has a unique zero, that zero lies in $[x,y]$. For $0\le x\le y\le \frac{1}{2}$, we have $F_\beta(1/2)\le 0$. Thus, from Lemma \ref{F_xi_lemma} $(iii)$, we have $m_{2,\beta}(3/4)<0$, for any $\beta>0$ which is contradicting \eqref{eqn_5_case2}. Hence, $\ker\Phi(z\,; x)=\{0\}$.
\end{pf}

\begin{prop}
    For the window function $g$ in \eqref{window}, $(\frac{1}{2},1)\in \mathcal{F}(g)$.
\end{prop}

\begin{proof}
    By Theorem \ref{rank equiv} and Lemma \ref{symmetry_singular_values}, it is sufficient to prove that for every $\vert{}z\vert{}=1$ and $0 \le x \le 1/2$, the $2 \times 1$ matrix
    $$\Phi(z\,;x) = \begin{pmatrix} \mathcal P_1(z\,;x) \\
    \mathcal P_1(-z\,;x) \end{pmatrix}$$
    has rank one. For simplicity, we write here $m_j = m_{j,1}$. Suppose to the contrary that $\mathcal P_1(z\,;x) =\mathcal P_1(-z\,;x) = 0$. Adding and subtracting these two equations and substituting $w = z^2$ yields the system
    \begin{align}\label{prop_eq1}
        m_0(x) + m_2(x)w + m_4(x)w^2 = 0
    \end{align}
    and
    \begin{align}\label{prop_eqn4}
        m_1(x) + m_3(x)w + m_5(x)w^2 = 0.
    \end{align}
    By Remark \ref{m1/m2_decreasing}, the ratio $m_1(x)/m_0(x)$ is decreasing on $(0,1]$
    At $\beta=1$, $R = e^{2\pi}$ gives
    $$\frac{m_1(1)}{m_0(1)} = \frac{R^4+1}{R^2} = R^2 + R^{-2} > 1.$$
    Hence, $m_1(x) > m_0(x)$ for $0 \le x \le 1$, where the $x=0$ case follows directly from $m_0(0) = 0 < m_1(0)$. Additionally, $m_0$ is increasing on $[0,1]$. 
In Lemma \ref{F_xi_lemma} $(iii)$ evaluated at $\beta=1$, we have $s = e^{\pi/2} > 7/5$. Since the unique zero $s_A$ occurring there satisfies $s_A < 7/5$, we obtain $F_1(1/2) > 0$. Thus, $m_2(1/2) > m_1(1/2) > 0$. By Lemma \ref{unique_Zero_m2}, $m_2$ possesses only one zero and remains strictly positive before it; consequently
    $m_2(x) > 0$, $0 \le x \le \frac{1}{2}$. 
    
    We now distinguish the possible values of $w$. If $w \notin \{-1, 1\}$, we divide \eqref{prop_eq1} by $w$, we obtain $m_2(x) + m_0(x)\overline{w} + m_4(x)w = 0$. Taking the imaginary parts gives
    $$\bigl(m_4(x) - m_0(x)\bigr)\Im w = 0.$$
    Since $0 \le x \le 1/2$, we have $1-x \ge x$. Applying \eqref{reflectio_on_m}, and the monotonicity of $m_0$ yields
    $$m_4(x) = m_1(1-x) > m_0(1-x) \ge m_0(x),$$
    which is impossible because $\Im w \neq 0$.

Now, suppose $w = 1$. Equation \eqref{prop_eq1} reduces to $m_0(x) + m_2(x) + m_4(x) = 0$, which is impossible because all three terms are nonnegative, $m_4(x)\ge m_0(x)\ge 0$ and $m_2(x) > 0$. 

It remains only to consider $w = -1$. In this case, \eqref{prop_eq1} becomes
\begin{align}\label{prop_eqn2}
    m_2(x) = m_0(x) + m_1(1-x).
\end{align}
Define $D(x) = m_2(x) - m_0(x) - m_1(1-x)$. Here, $R = e^{2\pi}$ and $Y = e^{2\pi x}$, we have $1 < Y \le \sqrt{R}$. For $1-x$, the corresponding exponential variable is $e^{2\pi(1-x)} = R/Y$. Therefore,
$$m_1(1-x) = \frac{P_1(R^2/Y) - C_1(R)P_1(R/Y)}{2R^3(R/Y)^3} = \frac{Y^6P_1(R^2/Y) - C_1(R)Y^6P_1(R/Y)}{2R^6Y^3},$$
where $P_1(Y)$ is defined in \eqref{P1_Y_eqn}. Consequently, we obtain from \eqref{new_m1_m2} that
\begin{align*}
    %2R^5Y^3D(x) =&P_1(R^2Y)-C_1(R)P_1(RY) +\bigl(C_2(R)-R^6\bigr)P_1(Y)-Y^6P_1(R^2/Y) +C_1(R)Y^6P_1(R/Y)\\
    D(x) =& \frac{(R^2+1)(Y^2-1)}{2R^5Y^3}\, H(R,Y),
\end{align*}
where
\begin{align} \label{prop_eqn3}
H(R,Y) =& R^8Y^2 -2R^7Y^3 -2R^7Y+R^6Y^4 +R^6 +2R^4Y^2 +R^2Y^4 +R^2 -2RY^3 -2RY+Y^2.
\end{align} 
An exact symbolic expansion is reproduced by the \href{https://github.com/Riya74012/Gabor-even-rational-windows-symbolic-verification}{MATLAB} code for verification.
For $0<x\le 1/2$, we have $1<Y\le\sqrt R$. Since all the
remaining terms in $H$ are positive,
\[
H(R,Y)>R^8Y^2-2(R^7+R)(Y^3+Y).
\]
As $R>1$, we have $R^7+R<2R^7$, while
\[
Y^3+Y
=
Y^2\left(Y+\frac1Y\right)
\le
Y^2(\sqrt R+1).
\]
Therefore,
\[
H(R,Y)
>
R^7Y^2\left[R-4(\sqrt R+1)\right].
\]
Here $R=e^{2\pi}$ and hence $\sqrt R=e^\pi>8$, so
$R-4(\sqrt R+1)>0$. Thus
$H(R,Y) > 0$ and, consequently $D(x)>0$. Therefore, \eqref{prop_eqn2} is impossible for $0 < x \le 1/2$. At $x=0$, we have $D(0) = 0$, but evaluating \eqref{prop_eqn4} with $w = -1$ gives
$$m_1(0) - m_3(0) + m_5(0) = m_1(0) - m_2(1) + m_0(1) > 0,$$
because $m_1(0) > 0$, $m_2(1) < 0$, and $m_0(1) > 0$. Hence, \eqref{prop_eq1} and \eqref{prop_eqn4} cannot vanish simultaneously at $x=0$.

Therefore, we have rigorously shown that $P_1(z;x)$ and $P_1(-z;x)$ never vanish simultaneously for $\vert{}z\vert{}=1$ and $0 \le x \le 1/2$. Thus $\operatorname{rank}\Phi(z;x) = 1$, throughout the required domain. By Theorem \ref{rank equiv}, $\mathcal{G}\left(g,\frac{1}{2},1\right)$ is a frame.

\end{proof}

\section{Non-Frame Obstructions}\label{Section_4}
\textbf{When $\alpha\beta=1$.}
From \eqref{final Phi_alpha matrix}, we have 
$$\Phi(z\,;\xi)=\mathcal{P}_\beta(z\,;\xi).$$ 
Further, from Remark \ref{Self_reciprocal_P}, we obtain that
$$\mathcal{P}_\beta(-1\,;1/2)=0$$
for every $\beta>0$. Consequently, therefore $\mathcal{G}(g, \alpha, \beta)$ is not a frame.

\begin{proof}[\textbf{Proof of Theorem \ref{non_frame_1}}]
Using the partial fraction representation of $g$ in \eqref{Partial_fraction}, together with 
    $$\sum_{k\in\Z}\frac{(-1)^k}{(t-k)^2+c^2}=\frac{\pi \cos{(\pi t)}\sinh{(\pi c)}}{c(\sinh^2{(\pi c)}+\sin^2{(\pi t)})},$$
    we obtain from the definition of the Zak transform that
    \begin{align}\label{Zak_g_in_Thm1.4}
        \mathcal{Z}_{g,\beta}\left(\frac{1}{6},\frac{1}{2}\right)&=\sum_{k\in \Z}(-1)^k\,g_{\beta}\left(\frac{1}{6}-k\right)\nonumber\\
        &= \frac{\pi\sqrt{3}\,\beta}{2} \sum_{j=1}^3 \frac{d_j}{j} \frac{\sinh(j\pi\beta)} {\sinh^2(j\pi\beta)+\frac{1}{4}},
    \end{align}
    where $d_1=-1/12$, $d_2=1/3$, and $d_3=-1/4$.
   Let $u=e^{-\pi \beta}\in (0,1)$. By expressing the hyperbolic functions in terms of $u$, we have
   $$\sinh(j\pi\beta) = \frac{1-u^{2j}}{2u^j} \quad \text{and} \quad \sinh^2(j\pi\beta)+\frac{1}{4} = \frac{1-u^{2j}+u^{4j}}{4u^{2j}}.$$
   Taking the ratio simplifies the terms inside the summation to
   $$\frac{\sinh(j\pi\beta)} {\sinh^2(j\pi\beta)+\frac{1}{4}} = \frac{2u^j(1-u^{2j})} {1-u^{2j}+u^{4j}}.$$
   Substituting this back into \eqref{Zak_g_in_Thm1.4} yields
   $$\mathcal{Z}_{g_\beta}\left(\frac{1}{6},\frac{1}{2}\right) = \frac{\pi\sqrt{3}\,\beta}{12} \left[ -\frac{u(1-u^2)}{u^4-u^2+1} +\frac{2u^2(1-u^4)}{u^8-u^4+1} -\frac{u^3(1-u^6)}{u^{12}-u^6+1} \right]=f_\beta(u)H_1(u),$$
   where
    $$f_\beta(u)=\frac{\pi \sqrt{3}\beta}{12}\frac{u(u-1)^3(u+1)(u^2+1)}{(u^4-u^2+1)(u^8-u^4+1)(u^{12}-u^6+1)}$$
    and 
    \begin{align}\label{H1_polynomial}
        H_1(u)=u^{16}-u^{14}-2u^{11}-4u^{10}-4u^9-3u^8-4u^7-4u^6-2u^5-u^2+1.
    \end{align}
    The \href{https://github.com/Riya74012/Gabor-even-rational-windows-symbolic-verification}{MATLAB} code reproduces an exact symbolic expansion for verification.

    For $0<u<1$, $f_\beta(u)\ne 0$ but $H_1(0)=1$ and $H_1(1)=-23$. Hence, there exists a $u_0\in (0,1)$ such that $H_1(u_0)=0$, \textit{i.e.}, $\beta_0=-\frac{\log u_0}{\pi}>0$.
    Numerically, $u_0\approx 0.608151$ and $\beta_0\approx 0.1583$. Thus, $\mathcal{Z}_{g_{\beta_0}}\left(\frac{1}{6},\frac{1}{2}\right)=0$. Using evenness $g_{\beta_0}$ and Zak quasi-periodicity, we obtain
    $$\mathcal{Z}_{g_{\beta_0}}\left(-\frac{1}{6},\frac{1}{2}\right)=0\text{ and }\mathcal{Z}_{g_{\beta_0}}\left(-\frac{1}{2},\frac{1}{2}\right)=0.$$\\
    \noindent
    $(i)$ \textbf{When }\bm{$a=\frac{1}{3}$}. For $p=1$, $q=3$, at $t=\frac{1}{6}$ and $ x=\frac{1}{2}$, the three entries of the Zibulski-Zeevi column vector vanish, 
    $$\mathcal{Z}_{g_{\beta_0}}\left(\frac{1}{6},\frac{1}{2}\right)=0, \quad\mathcal{Z}_{g_{\beta_0}}\left(-\frac{1}{6},\frac{1}{2}\right)=0,\text{ and }\mathcal{Z}_{g_{\beta_0}}\left(-\frac{1}{2},\frac{1}{2}\right)=0$$
    up to unimodular factors. Hence, $\rank \mathbf{M}\left(\frac{1}{6},\frac{1}{2}\right)=0$ and therefore, $\mathcal{G}(g_{\beta_0},\frac{1}{3},1)$ is not a frame. Consequently, the obstruction lattice is $P_{1/3}\approx(2.1056, 0.1583)$.\\
    \noindent
    $(ii)$ \textbf{When }\bm{$a=\frac{2}{3}$}. For $p=2$, $q=3$, at $t=\frac{5}{6}$ and $ x=0$.  Then the entries of the second column of $\mathbf{M}\left(\frac{5}{6},0\right)$ vanish
    $$\mathcal{Z}_{g_{\beta_0}}\left(\frac{5}{6},\frac{1}{2}\right)=-\mathcal{Z}_{g_{\beta_0}}\left(-\frac{1}{6},\frac{1}{2}\right)=0,$$
    $$\mathcal{Z}_{g_{\beta_0}}\left(\frac{1}{6},\frac{1}{2}\right)=0,\text{ and }\mathcal{Z}_{g_{\beta_0}}\left(-\frac{1}{2},\frac{1}{2}\right)=0.$$
  Hence, 
    $$\rank \mathbf{M}\left(\frac{5}{6},0\right)\le 1<2$$ 
    and therefore, $\mathcal{G}(g_{\beta_0},\frac{2}{3},1)$ is not a frame. Consequently, the obstruction lattice is $P_{2/3}\approx(4.2113, 0.1583)$.

Using Fourier duality and dilation, it is convenient to work with 
$$h_\lambda(\xi) = e^{-\lambda|\xi|} -2e^{-2\lambda|\xi|} +e^{-3\lambda|\xi|},$$
where $\lambda=\frac{2\pi\beta}{a}$.\\
\noindent
$(iii)$ \textbf{When }\bm{$a=\frac{1}{2}$}. Now we compute
\begin{align}\label{Zak_at_a_1/2}
    \mathcal{Z}h_\lambda\!\left(\frac{1}{4},\frac{1}{2}\right) &= \sum_{k\in\mathbb{Z}} (-1)^k h_\lambda\!\left(\frac{1}{4}-k\right)\nonumber\\
    &= \frac{e^{-\lambda/4} - e^{-3\lambda/4}}{1+e^{-\lambda}} - 2\frac{e^{-2\lambda/4} - e^{-6\lambda/4}}{1+e^{-2\lambda}} + \frac{e^{-3\lambda/4} - e^{-9\lambda/4}}{1+e^{-3\lambda}}.
\end{align}
Let
$ r=e^{-\lambda/4}=e^{-\pi\beta}\in(0,1)$. Then \eqref{Zak_at_a_1/2} becomes
\begin{align}
\mathcal{Z}h_\lambda\!\left(\frac{1}{4},\frac{1}{2}\right)&= \frac{r-r^3}{1+r^4} - 2\frac{r^2-r^6}{1+r^8} + \frac{r^3-r^9}{1+r^{12}}\nonumber\\
  &= -\frac{ r(r-1)^3(r+1)(r^2+1) }{ (r^4+1)(r^8+1)(r^8-r^4+1) } H_2(r),
\end{align}
where
$$H_{2}(r) =r^{12}-r^{10}-2r^9-3r^8-4r^7-4r^6-4r^5-3r^4-2r^3-r^2+1.$$
An exact symbolic expansion is reproduced by the \href{https://github.com/Riya74012/Gabor-even-rational-windows-symbolic-verification}{MATLAB} code for verification. Since $H_{2}(0)=1$ and $H_{2}(1)=-22$, there exists $r_0\in(0,1)$ such that
$H_{2}(r_0)=0.$
Let $\beta_0 = -\frac{\log r_0}{\pi}.$ Numerically, $r_0\approx 0.511123$ and $\beta_{0}\approx 0.213632$.
Hence,
$$ \mathcal Zh_\lambda\!\left(\frac14,\frac12\right)=0. $$
Since $h_\lambda$ is even,
$$ \mathcal Zh_\lambda\!\left(-\frac14,\frac12\right)=0. $$ 
For $p=1,q=2$, the column vector vanishes. Hence,
$a=\frac12$
is not a universal frame hyperbola. Consequently, the obstruction lattice is $P_{1/2}\approx(2.3405, 0.2136)$.\\
\noindent
$(iv)$ \textbf{When }\bm{$a=\frac{3}{4}$}. At $(t, \xi) = (\frac{1}{2},0)$, we have
$$\mathbf{M}\left(\frac{1}{2},0\right) = \begin{pmatrix}
\mathcal Zh_\lambda(\frac{1}{2},0) & \mathcal Zh_\lambda(\frac{1}{2},\frac{1}{3}) & \mathcal Zh_\lambda(\frac{1}{2},\frac{2}{3}) \\[0.25em] 
\mathcal Zh_\lambda(-\frac{1}{4},0) & \mathcal Zh_\lambda(-\frac{1}{4},\frac{1}{3}) & \mathcal Zh_\lambda(-\frac{1}{4},\frac{2}{3}) \\[0.25em]
\mathcal Zh_\lambda(-1,0) & \mathcal Zh_\lambda(-1,\frac{1}{3}) & \mathcal Zh_\lambda(-1,\frac{2}{3}) \\[0.25em]
\mathcal Zh_\lambda(-\frac{7}{4},0) & \mathcal Zh_\lambda(-\frac{7}{4},\frac{1}{3}) & \mathcal Zh_\lambda(-\frac{7}{4},\frac{2}{3}) \end{pmatrix}$$
Let $C_0, C_1$, and $C_2$ denote the three columns of this matrix and $\omega = e^{2\pi i/3}$. Now 
$$C_1 - \omega C_2 = \begin{pmatrix} \mathcal Zh_\lambda(\frac{1}{2},\frac{1}{3}) - \omega \mathcal Zh_\lambda(\frac{1}{2},\frac{2}{3}) \\[0.25em]
\mathcal Zh_\lambda(-\frac{1}{4},\frac{1}{3}) - \omega \mathcal Zh_\lambda(-\frac{1}{4},\frac{2}{3}) \\[0.25em] 
\mathcal Zh_\lambda(-1,\frac{1}{3}) - \omega \mathcal Zh_\lambda(-1,\frac{2}{3}) \\[0.25em]
\mathcal Zh_\lambda(-\frac{7}{4},\frac{1}{3}) - \omega \mathcal Zh_\lambda(-\frac{7}{4},\frac{2}{3}) \end{pmatrix}.$$
Using evenness and Zak quasi-periodicity, the first and third components vanish identically, \textit{i.e.}, 
$$(C_1 - \omega C_2)_0 = 0\qtq{and}(C_1 - \omega C_2)_2 = 0.$$
The remaining two components are obtained by direct geometric-series summation and have the form
$$(C_1 - \omega C_2)_1 = -(C_1 - \omega C_2)_3 = \Gamma(r)H_1(r),$$
where $r = e^{-\lambda/4}\in (0,1)$, $H_1(r)$ is a polynomial in $r$ defined in \eqref{H1_polynomial} and 
$$\Gamma(r) = \frac{ 2\sqrt{3}\,r(-\sqrt{3}+i)(r-1)^3(r+1)(r^2+1) }{ 4\displaystyle\prod_{\varepsilon=\pm1} (r^2+\varepsilon\sqrt{3}\,r+1) (r^4+\varepsilon\sqrt{3}\,r^2+1) (r^6+\varepsilon\sqrt{3}\,r^3+1) }$$
is a non-zero scaling factor. Hence, if $r=u_0$ is the root of $H_1$, then $C_1 = \omega C_2$.
Since the second and third columns are linearly dependent, we have
$$\operatorname{rank} \mathbf{M}\left(\frac{1}{2},0\right) \le 2 < 3.$$
Therefore, $a = \frac{3}{4}$ is not a universal frame curve. Here, $r=u_0\approx 0.60815$ implies that $\beta_0=-\frac{3}{2\pi}\log u_0\approx 0.23746$. Consequently, the obstruction lattice is $P_{3/4}\approx(3.1585, 0.23746)$.
\end{proof}

\begin{conjecture}\label{Conjecture}
 For every $p \in \mathbb{N}$, there exists $(\alpha_p, \beta_p) \in \mathbb{R}_+^2$ such that $\alpha_p\beta_p = \frac{p}{p+1}$ and $(\alpha_p, \beta_p) \notin \mathcal{F}(g)$.
\end{conjecture}
The above conjecture is true for $p=1,2,3$, which is proved analytically by the explicit obstruction arguments given in Theorem \ref{non_frame_1}. Now, we briefly describe the evidence supporting the preceding conjecture.

It is sufficient to construct, for each $p\ge 2$, a single value of $\beta > 0$ for which the Zibulski--Zeevi matrix associated with $\mathcal{G}(h_\lambda, a,1)$ loses rank. We choose $t = 0$ and $x = \frac{1}{2p}.$ At this point, the $(p+1) \times p$ matrix is given by
    $$M_{sr} = \mathcal Zh_\lambda\left( -a s, \frac{r+1/2}{p} \right), \quad 0 \le s \le p, \quad 0 \le r \le p-1.$$
    Define a $p\times p$, $U$ matrix with entries
    $$U_{rk} = \frac{1}{\sqrt{p}} e^{-2\pi i(r+1/2)k/p}, \quad 0 \le r, k \le p-1.$$
    The matrix $U$ is unitary. Let $B = MU$. Since $U$ is invertible,
\begin{align}\label{rank_B_and_rank_M}
    \operatorname{rank} B = \operatorname{rank} M.
\end{align}
The entries of $B$ are
\begin{align*}
    B_{sk}&= \frac1{\sqrt p} \sum_{r=0}^{p-1} Zh_\lambda\left( -as,\frac{r+1/2}{p} \right) e^{-2\pi i(r+1/2)k/p}\\
    &=\frac1{\sqrt p} \sum_{r=0}^{p-1} \sum_{m\in\mathbb Z} h_\lambda(-as-m) e^{2\pi i m(r+1/2)/p} e^{-2\pi i(r+1/2)k/p}\\
    &=\frac1{\sqrt p} \sum_{m\in\mathbb Z} h_\lambda(-as-m) e^{\pi i(m-k)/p} \sum_{r=0}^{p-1} e^{2\pi i r(m-k)/p}.
\end{align*}
The inner sum $\sum_{r=0}^{p-1} e^{2\pi i r(m-k)/p}$ is equal to $p$ precisely when $m\equiv k\pmod p$, and is zero otherwise. Substituting $m = k + np$ for $n \in \mathbb{Z}$ yields
\begin{align}\label{relation_B_and_H}
    B_{sk}= \sqrt{p} \sum_{n\in\mathbb{Z}} (-1)^n h_\lambda(-as-k-np)=\sqrt{p}\,H_{\lambda,p}(as+k),
\end{align}
where
$$H_{\lambda,p}(y) = \sum_{n\in\mathbb{Z}} (-1)^n h_\lambda(y+np).$$
For $c>0$, define
$$S_{c,p}(y)
=\sum_{n\in\mathbb Z} (-1)^n e^{-c|y+np|}.$$
For $0\le y\le p,$ we obtain
$$\sum_{n=0}^{\infty} (-1)^ne^{-c(y+np)} = \frac{e^{-cy}}{1+e^{-cp}}.$$
Therefore,
$$S_{c,p}(y) = \frac{e^{-cy}-e^{-c(p-y)}}{1+e^{-cp}}, \quad 0\le y\le p.$$
Furthermore, $S_{c,p}(p-y) = -S_{c,p}(y)$ and $S_{c,p}(y+p) = -S_{c,p}(y)$. Consequently, we have
$$H_{\lambda,p} = S_{\lambda,p} - 2S_{2\lambda,p} + S_{3\lambda,p}$$
with
\begin{align}\label{matrix_H}
    H_{\lambda,p}(p-y) = -H_{\lambda,p}(y)\qtq{and}H_{\lambda,p}(y+p) = -H_{\lambda,p}(y).
\end{align}
For $1\le s\le p$, we obtain from \eqref{matrix_H} that
\begin{align}\label{zeroth_col}
B_{p+1-s,0}
&=\sqrt p\,
H_{\lambda,p}
\bigl(a(p+1-s)\bigr)\nonumber\\
&=\sqrt p\,
H_{\lambda,p}(p-as)=-B_{s0}.
\end{align}
Now let $1\le k\le p-1$. Then, we obtain from \eqref{matrix_H} that
\begin{align}\label{middle_row}
    B_{p+1-s,k}=
\sqrt p\,
H_{\lambda,p}(p-as+k)=B_{s,p-k}.
\end{align}
Similarly
\begin{align}\label{zeroth_row}
B_{0,p-k}=\sqrt{p}H_{\lambda,p}(p-k)=-\sqrt{p}H_{\lambda,p}(k)=-B_{0k}.
\end{align}
Define an involution $T_c$ on the column space $\mathbb C^p$ by
$$T_ce_0=-e_0, \quad T_ce_k=e_{p-k}, \quad1\le k\le p-1,$$
and an involution $T_r$ on the row space $\mathbb C^{p+1}$ by
$$ T_rf_0=-f_0, \quad T_rf_s=f_{p+1-s}, \quad1\le s\le p.$$
Equations \eqref{zeroth_col}, \eqref{middle_row}, and \eqref{zeroth_row} imply
\begin{align}\label{TrB_BTc}
    T_rB=BT_c.
\end{align}
Thus $B$ maps each eigenspace of $T_c$ into the corresponding eigenspace of $T_r$. Now, we divide the proof based on whether $p$ is even or odd.\\
\noindent
\textbf{Case 1: $p$ is even.} Let $p=2m$, for $m\in \N$. Define the linear map $T_c:\mathbb C^p\to\mathbb C^p$ on the standard basis $e_0,\ldots,e_{p-1}$ by
$$ T_ce_0=-e_0, \quad T_ce_k=e_{p-k}, \quad 1\le k\le p-1.$$
Since $p=2m$, the index $m=p/2$ is fixed
$ T_ce_m=e_m. $
Also, for each
$ 1\le k\le m-1,$
the two indices $k$ and $p-k$ are distinct and are interchanged by $T_c$. Thus the $+1$-eigenspace of $T_c$ is
$$ E_c^+ = \{c\in\mathbb C^p:T_cc=c\}, $$
and has the orthonormal basis
\begin{align}\label{bases_Ec}
     u_k=\frac{e_k+e_{p-k}}{\sqrt2}, \quad 1\le k\le m-1, 
\end{align}
together with $u_m=e_m.$ Hence
$\operatorname{dim} E_c^+=m=\frac{p}{2}.$

Similarly, define $T_r:\mathbb C^{p+1}\to\mathbb C^{p+1}$ on the standard basis $f_0,\ldots,f_p$ by
$$ T_rf_0=-f_0, \quad T_rf_s=f_{p+1-s}, \quad1\le s\le p.$$
Since $p+1=2m+1$, the indices $1,\ldots,p$ split into the $m$ pairs
$ (s,p+1-s)$, $1\le s\le m. $
Therefore the $+1$-eigenspace
$$ E_r^+ = \{y\in\mathbb C^{p+1}:T_ry=y\} $$
has the orthonormal basis
\begin{align}\label{bases_Er}
     v_s= \frac{f_s+f_{p+1-s}}{\sqrt2}, \quad 1\le s\le m.
\end{align}
Thus, $\operatorname{dim} E_r^+=m=\frac{p}{2}$.

Let $B_p^+(\beta)$ be the matrix of the restriction $B: E_c^+ \longrightarrow E_r^+$ with respect to the bases \eqref{bases_Ec}, \eqref{bases_Er}. Using \eqref{middle_row} and \eqref{relation_B_and_H}, we have
\begin{align*} 
(B_p^+)_{s,k} &= v_s^* B u_k= \frac{1}{2} \left( B_{s,k} + B_{s,p-k} + B_{p+1-s,k} + B_{p+1-s,p-k} \right)\\
&=B_{s,k} + B_{s,p-k} \\
&=\sqrt{p} \left[ H_{\lambda,p}(as+k) + H_{\lambda,p}(as+p-k) \right],
\end{align*}
for $1 \le s \le m$ and $1 \le k < m$.
For the last column, $u_m = e_m$ and since $p-m = m$, we obtain from \eqref{middle_row} that
\begin{align*} 
(B_p^+)_{s,m} &= v_s^* B e_m= \frac{1}{\sqrt{2}} \left( B_{s,m} + B_{p+1-s,m} \right)\\
&= \sqrt{2p}\, H_{\lambda,p} \left(as+\frac{p}{2}\right). 
\end{align*}
Thus, $B_p^+(\beta)$ is an explicit real $m \times m$ matrix. Define $\Delta_p(\beta) = \det B_p^+(\beta)$. If $\Delta_p(\beta) = 0$, then there exists a nonzero $c \in E_c^+$ such that $B_p^+(\beta)c = 0$. By \eqref{TrB_BTc}, $Bc \in E_r^+$. But $B_p^+ c$ consists precisely of the coordinates of $Bc$ in the basis of $E_r^+$. Thus, $Bc = 0$ for $c\ne 0$. Therefore $\ker B \ne \{0\}$, so $\operatorname{rank} B < p$. Hence, we obtain from \eqref{rank_B_and_rank_M} that $\operatorname{rank} M < p.$\\
\noindent
\textbf{Case 2: $p$ is odd.} Let $p=2m-1$, for $m\in \N$. The $-1$ eigenspace of $T_c$ has dimension $m$, with orthonormal basis
\begin{align}\label{Tc_basis_for_podd}
 u_0 = e_0,\text{ and }  u_k = \frac{e_k-e_{p-k}}{\sqrt{2}}, 
\end{align}
for $1 \le k \le m-1$. The $-1$ eigenspace of $T_r$ also has dimension $m$, with basis
\begin{align}
   v_0 = f_0,\text{ and } v_s = \frac{f_s-f_{p+1-s}}{\sqrt{2}},
\end{align}
for $1 \le s \le m-1$. Let $B_p^-(\beta)$ denote the matrix of $B: E_c^- \longrightarrow E_r^-$. Then $ (B_p^-)_{0,0} = \sqrt{p}\,H_{\lambda,p}(0)$. Using \eqref{zeroth_row}, we obtain
\begin{align*}
    (B_p^-)_{0,k} &= \frac{1}{\sqrt{2}} \left(B_{0,k} - B_{0,p-k}\right)= \sqrt{2}\,B_{0,k}=\sqrt{2p}\,H_{\lambda,p}(k).
\end{align*}
for $1 \le k \le m-1$. Similarly, using \eqref{zeroth_row}, we obtain
\begin{align*}
     (B_p^-)_{s,0} &= \frac{1}{\sqrt{2}} \left( B_{s,0} - B_{p+1-s,0} \right)= \sqrt{2}\,B_{s0}= \sqrt{2p}\,H_{\lambda,p}(as),
\end{align*}
Finally, for $1 \le s, k \le m-1$,
\begin{align*} 
(B_p^-)_{s,k} &= \frac{1}{2}\bigl( B_{s,k} - B_{s,p-k} - B_{p+1-s,k} + B_{p+1-s,p-k} \bigr)\\ 
&= B_{s,k} - B_{s,p-k}, \end{align*}
where \eqref{middle_row} was used in the last equality. Therefore
$$(B_p^-)_{s,k} = \sqrt{p} \left[ H_{\lambda,p}(as+k) - H_{\lambda,p}(as+p-k) \right].$$
Define $\Delta_p(\beta) = \det B_p^-(\beta)$. Exactly as in the even case, $\Delta_p(\beta) = 0$ implies $\ker B \ne \{0\}$, and consequently $\operatorname{rank} M < p$.

Combining the two cases, define
$$\Delta_p(\beta) = \begin{cases} \det B_p^+(\beta), & p \text{ even},\\ \det B_p^-(\beta), & p \text{ odd}. \end{cases}$$
For every $p \ge 2$, $\Delta_p(\beta) = 0$ implies that $\mathcal{G}(h_\lambda,a,1)$ is not a frame. To provide numerical evidence for Conjecture \ref{Conjecture}, we now investigate positive zeros of \(\Delta_p\).

\subsection{Numerical evidence:} Numerical computations indicate that, for every
$ 2\le p\le200$, the function $\Delta_p$ possesses a positive zero. By the scaling parameter $\mu=\lambda p=2\pi(p+1)\beta$, we choose the two $p$-dependent parameters
$$ \beta_p^-=\frac{9}{20(p+1)},\qquad \beta_p^+=\frac{11}{20(p+1)}.$$
Then the computations consistently give
$$ \Delta_p(\beta_p^-) \Delta_p(\beta_p^+)<0, \qquad 2\le p\le200. $$
If this sign inequality could be established rigorously for all $2\le p\le200$, then continuity of $\Delta_p$ would immediately yield $\beta_p\in(\beta_p^-,\beta_p^+) $
such that $\Delta_p(\beta_p)=0.$
Therefore,
$$ \alpha_p=\frac{p}{(p+1)\beta_p} $$
would then give $\alpha_p\beta_p=\frac{p}{p+1}$, and $(\alpha_p,\beta_p)\notin\mathcal F(g).$
We have added \href{https://github.com/Riya74012/Gabor-even-rational-windows-symbolic-verification}{MATLAB} code for verification.

Interestingly, the determinant $\Delta_p$ may possess more than one positive zero. In particular, for $p=200$, numerical computation reveals a second sign-changing zero
$\beta_{200}^{(H)}\approx1.00996.$
Since
$a=\frac{200}{201}$,
the corresponding value $\alpha_{200}^{(H)} =\frac{200}{201\,\beta_{200}^{(H)}} \approx0.98521$
gives the numerical obstruction point
$$P_{200/201} \approx(0.98521,1.00996).$$
Thus, the computations provide evidence that non-frame obstruction points may occur even in the region $\beta>1$, while $\alpha\beta<1$.

\section*{Data Availability}
The author generated the figures using MATLAB. The data that support the findings of this study are available within the article and its supplementary material.
\section*{Declarations}
\textbf{Conflict of Interest}: The author has no conflicts to disclose.

\bibliographystyle{plain}

\bibliography{totallypositive}
\end{document}